\documentclass[a4paper,10pt,reqno]{amsart}
\usepackage[latin1]{inputenc}
\usepackage{amssymb}
\usepackage{amsfonts}
\usepackage{amscd}
\usepackage{mathrsfs}
\usepackage[dvips]{graphicx}
\usepackage[T1]{fontenc}
\usepackage[english]{babel}
\usepackage{marvosym}
\usepackage{amsthm}
\usepackage{multicol}
\usepackage[all]{xy}

\usepackage{wasysym}

\theoremstyle{plain}
\newtheorem{thm}{Theorem}[section]
\newtheorem{lem}[thm]{Lemma}

\newtheorem{coro}[thm]{Corollary}

\theoremstyle{definition}
\theoremstyle{remark}
\newtheorem{rk}[thm]{Remark}
\newtheorem{df}[thm]{Definition}

\numberwithin{equation}{section}

\def\ui{{\mathbf{i}}}
\def\uj{{\mathbf{j}}}

\def\bbC{\mathbb{C}}

\def\bbN{\mathbb{N}}

\def\bbQ{\mathbb{Q}}

\def\bbZ{\mathbb{Z}}

\def\scrL{\mathscr{L}}

\def\frakl{\mathfrak{l}}

\def\frakS{\mathfrak{S}}

\def\calC{\mathcal{C}}

\def\calH{\mathcal{H}}

\def\calK{\mathcal{K}}

\def\calO{\mathcal{O}}
\def\calP{\mathcal{P}}

\def\frakb{\mathfrak{b}}

\def\frakg{\mathfrak{g}}
\def\frakh{\mathfrak{h}}
\def\frakl{\mathfrak{l}}

\def\frakp{\mathfrak{p}}

\def\bfa{\mathbf{a}}

\def\bfd{\mathbf{d}}
\def\bfe{\mathbf{e}}

\def\bfk{\mathbf{k}}
\def\bfm{\mathbf{m}}

\def\bfs{\mathbf{s}}

\def\bfA{\mathbf{A}}

\def\bfP{\mathbf{P}}

\def\homo{\operatorname{\it \mathscr{H}\kern-.25em om}}
\def\ext{\operatorname{\it \mathscr{E}\kern-.25em xt}}
\def\edo{\operatorname{\it \mathscr{E}\kern-.25em nd}}
\def\der{\operatorname{\it \mathscr{D}\kern-.25em er}}

\def\mod{\mathrm{mod}}
\def\grmod{\mathrm{grmod}}

\def\Hom{\operatorname{Hom}\nolimits}

\def\End{\operatorname{End}\nolimits}

\def\soc{{\operatorname{soc}\nolimits}}
\def\rad{{\operatorname{rad}\nolimits}}

\def\Ext{\operatorname{Ext}\nolimits}

\def\Ker{\operatorname{Ker}\nolimits}

\def\Im{\operatorname{Im}\nolimits}

\def\Id{\operatorname{Id}\nolimits}

\def\IC{\mathrm{IC}}

\author{Ruslan Maksimau}
\address{Universit\' e Paris Diderot - Paris 7, Institut de Math\' ematiques de Jussieu - Paris Rive Gauche,
            UMR7586,
            B\^atiment Sophie Germain,
            Case 7012,
            75205 PARIS Cedex 13}
\email{ruslmax@gmail.com, maksimau@math.jussieu.fr}

\title
[Quiver Schur algebras and Koszul duality]
{Quiver Schur algebras and Koszul duality}

\begin{document}
\begin{abstract}
We prove that the category of graded finitely generated representations of the the cyclotomic quiver Schur algebra is a Koszul category.
\end{abstract}

\maketitle

\setcounter{tocdepth}{2}

\tableofcontents

\section{Introduction}
Let $e$, $l$ be integers, $e>1$, $l>0$. Let $\Gamma$ be the cyclic quiver of type $\widehat{A}_{e-1}$ and let $I$ be the set of its vertices. To each $d\in \bbN$ and each tuple $\bfs\in (\bbZ/e\bbZ)^l$ we associate a cyclotomic $q$-Schur algebra $S^\bfs_d$. This algebra admits a block decomposition $S^\bfs_d=\bigoplus_{\bfd}S^\bfs_\bfd$, where $\bfd$ runs over the set of elements $\bfd=\sum_{i\in I}d_i\cdot i\in\bbN I$ such that $\sum_{i\in I}d_i=d$. The algebra $S^\bfs_\bfd$ is quasi-hereditary and the standard modules are the Weyl modules. The Weyl modules are labelled by a subset $\calP^{l}_\bfd$ of the set $\calP^{l}_d$ of $l$-partitions of $d$.
By \cite{R}, \cite{RSVV} we can identify the category $\mod(S^{\bfs}_{\bfd})$ of finite dimensional $S_{\bfd}^{\bfs}$-modules with a highest weight subcategory $\bfA$ of a parabolic category $\calO$ of $\widehat{\mathfrak{gl}_N}$ at negative level for some $N\in\bbN$, see Theorem \ref{thm_equiv-Schur-O-ungr}. The category $\bfA$ has a Koszul grading by \cite[Thm.~6.4]{SVV2}. In Section \ref{subs_admis-grad} we define a class of gradings on $S^\bfs_\bfd$ that we call \emph{admissible gradings}. An example of such grading is given in \cite{SW}. We fix an admissible grading of $S^\bfs_\bfd$. It yields a grading of the category $\mod(S^{\bfs}_{\bfd})$. It is natural to ask whether this grading coincides with the Koszul grading of $\bfA$.

In Section \ref{subs_isom-gr-KLR} we define a basic graded algebra $^bS^\bfs_{\bfd}$ which is graded Morita equivalent to $S^\bfs_{\bfd}$.
Let $S^\calO$ be the endomorphism algebra of the minimal projective generator of $\bfA$ (equipped with the Koszul grading). We will prove the following result, see Theorem \ref{thm_main-theorem}.
\begin{thm}
\label{thm_main-thm}
There exists a graded algebra isomorphism $^bS_{\bfd}^{\bfs}\simeq S^\calO$. \qed
\end{thm}
The proof of the theorem is inspired by \cite{HM} and \cite{SW}. It is essential to calculate the graded multiplicities of the simple modules in the parabolic Verma modules in $\calO$. This is done in Appendix \ref{app_mult}.

After this paper was written we received a copy of Webster's paper \cite{Web2} where a similar statement as Theorem \ref{thm_main-thm} is announced.

\medskip

\section{Quiver Schur algebras and Koszul duality}
Let $\bfk$ be a field of characteristic zero. All functors between additive (resp. $\bfk$-linear) categories are supposed to be additive (resp. $\bfk$-linear). Let $e$, $l$ be integers, $l>0$, $e>1$. Fix a tuple $\bfs=(s_1,\cdots,s_l)$ of elements of $\bbZ/e\bbZ$. For each $\bbZ$-graded finite dimensional $\bfk$-vector space $V$, let $\dim_q V\in\bbN[q,q^{-1}]$ be its graded dimension, i.e., $\dim_q V=\sum_{g\in \bbZ}(\dim V_g) q^g$.

\medskip

\subsection{Graded categories}
For any noetherian ring $A$, let $\mod(A)$ be the category of finitely generated left $A$-modules. For any noetherian $\bbZ$-graded ring $A$, let $\grmod(A)$ be the category of $\bbZ$-graded finitely generated left $A$-modules. The morphisms in $\grmod(A)$ are the morphisms which are homogeneous of degree zero. 

\begin{df}
A $\bbZ$-\emph{category} (or a \emph{graded category}) is an additive category $\overline\calC$ with a fixed autoequivalence $T\colon\overline\calC\to \overline\calC$. We call $T$ the shift functor. For each $X\in\overline \calC$ and $n\in\bbZ$, we set $X\langle n \rangle=T^n(X)$. A functor of $\bbZ$-categories is a functor commuting with the shift functor.
\end{df}

For a graded noetherian ring $A$ the category $\grmod(A)$ is a $\bbZ$-category where $T$ is the shift of grading, i.e., for $M=\oplus_{n\in\bbN}M_n\in \grmod(A), ~k\in\bbZ$, we have $T(M)_k=M_{k-1}$.


\begin{df}
\label{def_gr-ver}
Let $\calC$ be an abelian category. We say that an abelian $\bbZ$-category $\overline \calC$ is a \emph{graded version} of $\calC$ if there exists a functor $F_{\calC}\colon\overline \calC\to\calC$ and a graded noetherian ring $A$ such that we have the following commutative diagram, where the horizontal arrows are equivalences of categories and the top horizontal arrow is a functor of $\bbZ$-categories
$$
\begin{CD}
\overline \calC\ @>>> \grmod(A)\\
@V{F_\calC}VV                 @V\text{forget}VV\\
\calC            @>>> \mod(A).
\end{CD}
$$
\end{df}
In the setup of Definition \ref{def_gr-ver}, we say that the object $\overline X\in\overline\calC$ is a \emph{graded lift} of the object $X\in\calC$ if we have $F_\calC(\overline X)\simeq X$. For objects $X, Y\in\calC$ with fixed graded lifts $\overline{X},\overline{Y}$ the $\bbZ$-module $\Hom_\calC(X,Y)$ admits a $\bbZ$-grading given by $\Hom_\calC(X,Y)_n=\Hom_{\overline\calC}(\overline X\langle n\rangle,\overline Y)$. In the sequel we will often denote the object $X$ and its graded lift $\overline X$ by the same symbol.

\begin{df}
We say that an abelian category is \emph{gradable} if it has a graded version in sense of Definition \ref{def_gr-ver}. For two gradable abelian categories $\calC_1$, $\calC_2$ with graded versions $\overline{\calC_1}$, $\overline{\calC_2}$ we say that the functor of $\bbZ$-categories $\overline\Phi\colon\overline\calC_1\to\overline\calC_2$ is a \emph{graded lift} of a functor  $\Phi\colon\calC_1\to\calC_2$ if $F_{\calC_2}\circ\overline\Phi=\Phi\circ F_{\calC_1}$.
We say that $\calC_1$ and $\calC_2$ are \emph{equivalent as gradable categories} if there exists a commutative diagram
$$
\begin{CD}
\overline{\calC_1} @>>> \overline{\calC_2}\\
@V{F_{\calC_1}}VV       @V{F_{\calC_2}}VV\\
\calC_1            @>>> \calC_2\\
\end{CD}
$$
such that the horizontal arrows are equivalences of categories and the top horizontal arrow is a functor of $\bbZ$-categories.
\end{df}

\medskip

\subsection{Koszul duality}
For an abelian category $\bfA$, let $D^b(\bfA)$ denote the associated bounded derived category. If $\bfA$ is an abelian category with a graded version $\overline{\bfA}$, let $[\overline\bfA]$ be the Grothendieck group of $\overline\bfA$, i.e., the $\bbZ$-module with generators $[M]$, $M\in \overline\bfA$, and relations
$$
[X]+[Z]=[Y] \quad\forall \text{ exact sequence } 0\to X\to Y\to Z\to 0 \text{ in }\overline\bfA.
$$
It admits a $\bbZ[q,q^{-1}]$-module structure given by
$$
q[M]=[M\langle 1 \rangle] \quad \forall M\in \overline\bfA.
$$
Note that $[\overline\bfA]$ coincides with the Grothendieck group of the triangulated category $D^b(\overline\bfA)$.

Let $A=\bigoplus_{n\geqslant 0}A_n$ be an $\bbN$-graded $\bfk$-algebra such that $A_0$ is semisimple. We identify $A_0$ with the left graded $A$-module $A_0\simeq A/{\oplus_{n>0}A_n}$.
\begin{df}
The graded algebra $A$ is \emph{Koszul} if the left graded $A$-module $A_0$ admits a projective resolution $\cdots\to P^2\to P^1\to P^0\to A_0$ such that $P^k$ is generated by its degree $k$ component.

\medskip
If $A$ is Koszul, we consider the graded $\bfk$-algebra $A^!=\Ext^*_A(A_0,A_0)^{\rm op}$ and we call it the \emph{Koszul dual} algebra to $A$. Here $(\bullet)^\mathrm{op}$ means the algebra with the opposite product. If $\overline\bfA=\grmod(A)$, we write $\overline\bfA^!=\grmod(A^!)$ and we call it the \emph{Koszul dual category} of $\overline\bfA$. We write also $\bfA^!=\mod(A^!)$.
\end{df}

For a graded noetherian ring $A$ we abbreviate $\overline{D}^b(A)=D^b(\grmod(A))$.
Let $[\bullet]$ denote the shift of the cohomological degree.
The following is well-known, see \cite[Thm.~2.12.5,~2.12.6]{BGS}.

\begin{thm}
\label{thm_Koszul-duality}
Let $A$ be a Koszul $\bfk$-algebra. Assume that $A$ and $A^!$ are finite dimensional. Then, the following hold.

{\rm (a)} The algebra $A^!$ is also Koszul and there is a graded algebra isomorphism $(A^!)^!\simeq A$.

{\rm (b)} There is an equivalence of categories
$
K\colon \overline{D}^b(A)\to \overline{D}^b(A^!)
$
such that $K(M \langle n \rangle)=K(M)\langle -n \rangle[-n]$.
 \qed
\end{thm}

For a $\bbZ[q,q^{-1}]$-module $L$, let $\widetilde L$ be the $\bbZ[q,q^{-1}]$-module isomorphic to $L$ as a $\bbZ$-module such that $q$ acts on $\widetilde L$ in the same way as $-q^{-1}$ acts on $L$.
\begin{coro}
\label{coro_Groth-isom-Koszul}
Assume that $A$ is Koszul and that $A$, $A^!$ are finite dimensional. There is a $\bbZ[q,q^{-1}]$-module isomorphism $[\grmod(A)]\to\widetilde{[\grmod(A^!)]}$, $[M]\mapsto [K(M)]$, $\forall M\in \overline{D}^b(A)$,
where $K$ is as in Theorem \ref{thm_Koszul-duality}.
\qed
\end{coro}

\medskip

\subsection{Graded highest weight categories}
\label{subs_graded-hw-cat}
Let $\calC$ be an abelian category which is equivalent to the category $\mod(A)$ for a finite dimensional $\bfk$-algebra $A$ and let $(\Lambda,\leqslant)$ be a finite poset. Let $\Delta=\{\Delta(\lambda);\lambda\in\Lambda\}$ be a family of objects in $\calC$.
\begin{df}
The pair $(\calC,\Delta)$ is a \emph{highest weight category} if

(a) if $\Hom_\calC(\Delta(\lambda),M)=0$ for each $\lambda\in \Lambda$, then $M=0$,

(b) for each $\lambda\in\Lambda$, there is a projective object $P(\lambda)\in\calC$ and a surjection $f\colon P(\lambda)\to\Delta(\lambda)$ such that $\Ker(f)$ has a (finite) filtration whose successive quotients are objects of the form $\Delta(\mu)$ with $\mu>\lambda$,

(c) for each $\lambda\in\Lambda$, we have $\End_\calC(\Delta(\lambda))=\bfk$,

(d) for each $\lambda,\mu\in\Lambda$ such that $\Hom_{\calC}(\Delta(\lambda),\Delta(\mu))\ne 0$, we have $\lambda\leqslant\mu$.

\noindent We call the objects $\Delta(\lambda)$'s the \emph{standard objects}.
\end{df}

For each $\lambda\in\Lambda$, the module $\Delta(\lambda)$ has a unique simple quotient $L(\lambda)$. Let $\nabla(\lambda)$ be the costandard object with parameter $\lambda$, see \cite[Prop.~2.1]{RSVV}.

Now, assume that the $\bfk$-algebra $A$ is $\bbZ$-graded.
In addition, we make the following assumption :

\smallskip

$(\text{a}_1)$ for each $\lambda\in\Lambda$, the modules $\Delta(\lambda)$, $\nabla(\lambda)$, $P(\lambda)$, $L(\lambda)$ admit graded lifts.

\smallskip

By definition, we have $\dim\Hom_\calC(\Delta(\lambda),L(\lambda))=\dim\Hom_\calC(P(\lambda),L(\lambda))=1$ for each $\lambda\in\Lambda$. Moreover, we have $\dim\Hom_{\calC}(L(\lambda),\nabla(\lambda))=1$, see \cite[Sec.~A.1]{Don}. We fix graded lifts of $P(\lambda)$, $\Delta(\lambda)$, $L(\lambda)$, $\nabla(\lambda)$ such that each morphism $P(\lambda)\to \Delta(\lambda)$, $\Delta(\lambda)\to L(\lambda)$, $L(\lambda)\to\nabla(\lambda)$ is homogeneous of degree $0$.

Set $\overline\calC=\grmod(A)$. Let $M,~L\in\overline\calC$, with $L$ irreducible. Let $[M:L]$ be the multiplicity of $L$ in a (graded) Jordan-H\"older series of $M$. Write
$$
[M:L]_q=\sum_{k\in\bbZ}[M:L\langle k\rangle]q^k\in\bbN[q,q^{-1}].
$$

Let $\overline\calC^\Delta$ be the full subcategory of objects of $\overline\calC$ that have a finite filtration by the graded modules of the form $\Delta(\lambda)\langle k\rangle$ for $\lambda\in\Lambda$, $k\in\bbZ$. We will refer to such filtration as a \emph{graded $\Delta$-filtration}. Assume that $M\in\overline\calC^\Delta$. Let $(M:\Delta(\lambda)\langle k\rangle)$ be the multiplicity of $\Delta(\lambda)\langle k\rangle$ in a graded $\Delta$-filtration of $M$. A standard argument shows that this multiplicity is independent of the choice of a graded $\Delta$-filtration. Set also
$$
(M:\Delta^\lambda)_q=\sum_{k\in\bbZ}(M:\Delta^\lambda\langle k\rangle)q^k\in\bbN[q,q^{-1}].
$$

\begin{lem}
\label{lem_Ext-delta-filtr}
Assume that $(\operatorname{a}_1)$ holds. An object $M\in\overline\calC$ is graded $\Delta$-filtered if and only if
$\Ext^1_{\calC}(M,\nabla(\lambda))=0$ for each $\lambda\in \Lambda$.
\end{lem}

\begin{proof}
The only if part follows from \cite[Prop.~A2.2~(iii)]{Don}. Now, let us prove the if part. Assume that $\Ext^1_\calC(M,\nabla(\lambda))=0$ for each $\lambda\in \Lambda$. We want to prove that $M\in\overline\calC^\Delta$. We will prove the statement by induction on the length of a Jordan-H\"older series of $M$. If $M=0$, the statement is trivial. Now, assume that $M\ne 0$. Let $\lambda_0\in\Lambda$ be minimal such that $\Hom_\calC(M,L(\lambda_0))\ne 0$. Let $k_0\in\bbZ$ be such that $\Hom_{\overline\calC}(M,L(\lambda_0)\langle k_0 \rangle)\ne 0$. First, we claim that for each $\mu\leqslant\lambda_0$ we have $\Ext^1_\calC(M,L(\mu))=0$. Really, consider the short exact sequence
$$
0\to L(\mu)\to \nabla(\mu)\to\nabla(\mu)/L(\mu)\to 0
$$
in $\calC$. It yields an exact sequence
$$
\Hom_\calC(M,\nabla(\mu)/L(\mu))\to\Ext^1_\calC(M,L(\mu))\to \Ext^1_\calC(M,\nabla(\mu)).
$$
The rightmost term is zero by assumption. Moreover, each simple subquotient of $\nabla(\mu)/L(\mu)$ is of the form $L(\mu')$ with $\mu'<\mu\leqslant \lambda_0$. Thus the leftmost term is also zero by minimality of $\lambda_0$. This implies  $\Ext^1(M,L(\mu))=0$.

Let $N(\lambda_0)$ be the radical of $\Delta(\lambda_0)$.
Consider the short exact sequence
$$
0\to N(\lambda_0)\langle k_0 \rangle\to \Delta(\lambda_0)\langle k_0 \rangle\to L(\lambda_0)\langle k_0 \rangle\to 0
$$
in $\overline\calC$. It yields the exact sequence
$$
\Hom_{\overline\calC}(M,\Delta(\lambda_0)\langle k_0 \rangle)\stackrel{\delta}{\to} \Hom_{\overline\calC}(M,L(\lambda_0)\langle k_0 \rangle)\to \Ext^1_{\overline\calC}(M,N(\lambda_0)\langle k_0 \rangle).
$$
Each simple subquotient of $N(\lambda_0)$ in $\calC$ is of the form $L(\mu)$ with $\mu<\lambda_0$. By the previous discussion this implies $\Ext^1_\calC(M,N(\lambda_0))=0$. In particular, we deduce that the rightmost term is zero. Thus the morphism $\delta$ is surjective. Moreover, $\Hom_{\overline\calC}(M,L(\lambda_0)\langle k_0 \rangle)\ne 0$ by the choice of $\lambda_0$ and $k_0$. Let $f\in \Hom_{\overline\calC}(M,\Delta(\lambda_0)\langle k_0 \rangle)$ be such that $\delta(f)\ne 0$. The image of $f$ is not contained in $N(\lambda_0)$. Thus $f$ is surjective. Consider the exact sequence
$$
0\to\Ker(f)\to M\to \Delta(\lambda_0)\to 0
$$
in $\calC$.
For each $\mu\in \Lambda$ it yields an exact sequence
$$
\Ext^1_{\calC}(M,\nabla(\mu))\to\Ext^1_{\calC}(\Ker(f),\nabla(\mu))\to \Ext^2_{\calC}(\Delta(\lambda_0),\nabla(\mu)).
$$
The leftmost term is zero by assumption and the rightmost term is zero by \cite[Prop.~A2.2~(ii)]{Don}. Thus, we have $\Ext^1_{\calC}(\Ker(f),\nabla(\mu))=0$ for each $\mu\in\Lambda$. Applying the induction hypothesis to $\Ker(f)$ we get $\Ker(f)\in\overline\calC^\Delta$. This implies $M\in\overline\calC^\Delta$.
\end{proof}

\smallskip

\begin{rk}
Lemma \ref{lem_Ext-delta-filtr} and \cite[Prop.~A2.2~(iii)]{Don} imply that under the assumption
$(\operatorname{a}_1)$ an object $M\in\overline\calC$ admits a graded $\Delta$-filtration if and only if it admits a $\Delta$-filtration as an ungraded object.
\end{rk}

\smallskip

\begin{coro}
Assume that $(\operatorname{a}_1)$ holds. Then we have  $P(\lambda)\in\overline\calC^\Delta$ for each $\lambda\in\Lambda$.
\qed
\end{coro}

\smallskip

\begin{lem}
\label{lem_grBGG1}
Assume that $(\operatorname{a}_1)$ holds.
Then we have $[\nabla(\lambda):L(\mu)]_{q^{-1}}=(P(\mu):\Delta(\lambda))_q$ for each $\lambda,\mu\in\Lambda$.
\end{lem}

\smallskip

\begin{proof}
For each $\lambda,\mu\in\Lambda$ and each $k\in\bbZ$, we have $\dim_q\Hom_{\calC}(P(\lambda),L(\mu)\langle k \rangle)=\delta_{\lambda,\mu}q^{k}$. Moreover, the functor $\Hom_{\calC}(P(\lambda),\bullet)\colon\overline\calC\to\grmod(\bfk)$ is exact. Hence, for each $M\in\overline\calC$, we have $[M:L(\lambda)]_q=\dim_{q}\Hom_{\calC}(P(\lambda),M)$.

Next, we have $\dim_q\Hom_\calC(\Delta(\lambda)\langle k \rangle,\nabla(\mu))=\delta_{\lambda,\mu}q^{-k}$, see \cite[Prop.~2.1]{RSVV}. Moreover, the functor $\Hom_\calC(\bullet,\nabla(\mu))\colon \overline\calC^\Delta\to \grmod(\bfk)$ is exact because $\Ext^1_{\calC}(\Delta(\lambda),\nabla(\mu))=0$ for each $\lambda,\mu\in\Lambda$. Thus, for each $M\in\overline\calC^\Delta$, we have
$(M:\Delta(\lambda))_q=\dim_{q^{-1}}\Hom_{\calC}(M,\nabla(\lambda))$.
Therefore, for each $\lambda,\mu\in\Lambda$, we get
$$
[\nabla(\lambda):L(\mu)]_{q^{-1}}=\dim_{q^{-1}}\Hom_{\calC}(P(\mu),\nabla(\lambda))=(P(\mu):\Delta(\lambda))_q.
$$
\end{proof}

Now, we add one more assumption:

\smallskip

$(\text{a}_2)$ there is an equivalence of categories $D\colon\calC^{\rm op}\to\calC$ such that
\begin{itemize}
    \item[\textbullet] $D$ admits a graded lift $\overline D\colon\overline\calC^{\rm op}\to\overline\calC$ that is also an equivalence of categories,
    \item[\textbullet] $D^2=\Id_{\calC}$, $\overline D^2=\Id_{\overline\calC}$,
    \item[\textbullet] $\overline D(L(\lambda))=L(\lambda)$ for each $\lambda\in\Lambda$,
    \item[\textbullet] $\overline D(M\langle 1\rangle)=D(M)\langle -1\rangle$ for each $M\in\overline\calC$.
\end{itemize}

\smallskip

\begin{coro}
\label{coro_grBGG2}
Assume that $(\operatorname{a}_1)$ and $(\operatorname{a}_2)$ hold.
Then, for each $\lambda,\mu\in\Lambda$, we have $[\Delta(\lambda):L(\mu)]_q=(P(\mu):\Delta(\lambda))_q$.
\end{coro}
\begin{proof}
By \cite[Prop.~2.6]{RSVV} the functor $D$ must exchange $\Delta(\lambda)$ and $\nabla(\lambda)$ for each $\lambda\in\Lambda$ because the category $\calC^{\rm op}$ is a highest weight category with standard modules $\nabla(\lambda)$, $\lambda\in\Lambda$, see the discussion after \cite[Lem.~A3.5]{Don}. Thus by the normalization of grading the functor $\overline D$ exchanges the graded objects $\Delta(\lambda)$ and $\nabla(\lambda)$. For each $\lambda,\mu\in\Lambda$ and each $k\in\bbZ$, we have
$$
[\Delta(\lambda):L(\mu)\langle k \rangle]=[\overline D(\nabla(\lambda)):\overline D(L(\mu)\langle -k \rangle)]=[\nabla(\lambda):L(\mu)\langle -k \rangle].
$$
This implies $[\Delta(\lambda):L(\mu)]_q=[\nabla(\lambda):L(\mu)]_{q^{-1}}$. Now, apply Lemma \ref{lem_grBGG1}.
\end{proof}


\medskip

\subsection{Combinatorics}
\label{subs_comb}

A \emph{composition} of an integer $n\geqslant 0$ is a tuple of positive integers $(\lambda_1,\cdots,\lambda_k)$ such that $\sum_{t=1}^k\lambda_t=n$. Note that this is not the usual definition of a composition because we do not allow zero components. We set $l(\lambda)=k$, $|\lambda|=n$.
A \emph{partition} of an integer $n\geqslant 0$ is a composition $(\lambda_1,\cdots,\lambda_k)$ of $n$ such that $\lambda_1\geqslant \lambda_2\geqslant\cdots\geqslant \lambda_k$. To each partition of $n$ we associate a Young diagram $Y(\lambda)$ such that $Y(\lambda)$ has $\lambda_t$ boxes in the $t$th row for each $t\in [1,k]$.  The empty partition is represented by the empty diagram.
An $l$-\emph{composition} (resp. $l$-\emph{partition}) of an integer $n\geqslant 0$ is an $l$-tuple $\lambda=(\lambda^{1},\cdots,\lambda^{l})$ of compositions (resp. partitions) of integers $n_1,\cdots,n_l\geqslant 0$ such that $\sum_{t=1}^ln_t=n$. We associate with each $l$-partition an $l$-tuple of Young diagrams $Y(\lambda)=(Y(\lambda^1),\cdots,Y(\lambda^l))$.
Let $C^l_n$ (resp. $\calP^l_n$) be the set of $l$-compositions (resp. $l$-partitions) of $n$.
Set $\calP^l=\coprod_{n\in\bbN}\calP^l_n$.

Let $\Gamma$ be the quiver with the vertex set $I=\bbZ/e\bbZ$ and the arrow set $H=\{i\to i+1; i\in I\}$. Let $(a_{ij})_{i,j\in I}$ be the associated Cartan matrix, i.e., we put $a_{i,j}=2\delta_{i,j}-\delta_{i,j+1}-\delta_{i,j-1}$.
For $\bfd=\sum_{i\in I}d_i\cdot i$ in $\bbN I$ set$|\bfd|=\sum_{i\in I}d_i$. Let $I^\bfd$ be the set of tuples $\ui=(i_1,\cdots,i_{|\bfd|})$ in $I^{|\bfd|}$ such that $\ui$ has $d_i$ entries equal to $i$ for each $i\in I$. We have $I^d=\coprod_{|\bfd|=d}I^\bfd$.
Let $\frakS_d$ be the symmetric group of rank $d$. It acts on $I^d$ by permutation of the components. This action preserves each subset $I^\bfd\subset I^d$ with $|\bfd|=d$.

\begin{df}
Let $\lambda\in \calP^l$. The \emph{residue} $r(b)\in I$ of the box $b\in Y(\lambda)$ situated in the spot $(i,j)$ of the $p$th diagram is $s_p+j-i$ mod $e$. Write
$
\calP^l_\bfd=\{\lambda\in\calP^l; \sum_{b\in Y(\lambda)}r(b)=\bfd\},
$
where the sum is taken in $\bbN I$.
\end{df}

\medskip

\subsection{Rigid modules}
\label{subs_ridid-modules}
Let $A$ be a noetherian $\bfk$-algebra and let $M$ be an $A$-module of finite length.
Let
$
M\supset\rad^1 M\supset\rad^2M\supset\cdots\rad^rM=0
$
be the \emph{radical filtration} of $M$, i.e., $\rad^1M=\rad M$ and $\rad^{a+1}M=\rad(\rad^a M)$, for $a\geqslant 1$, $\rad^{r-1} M\ne 0$. Let
$
M=\soc^s M\supset\cdots\soc^2 M\supset\soc^1M\supset 0
$
be the \emph{socle filtration} of $M$, i.e., $\soc^1M=\soc M$ and $\soc^{a+1}M$ is the inverse image of $\soc(M/\soc^a M)$ under the natural morphism $M\to M/\soc^aM$ for $a\geqslant 1$, $\soc^{s-1} M\subsetneq M$.
An $A$-module $M$ of finite length is \emph{rigid} if its radical and socle filtrations coincide, i.e., $r=s$ and $\rad^a M=\soc^{r-a} M$ for each $a\in[1,r-1]$.

Now, assume that $A$ is an $\bbN$-graded $\bfk$-algebra and let $M$ be a graded $A$-module of finite length. Let $z\in\bbZ$ be minimal such that $M_z\ne 0$.
The \emph{grading filtration} of $M$ is the filtration
$
M=\mathcal{G}r_{z}M \supset\mathcal{G}r_{z+1} M\supset\cdots
$
such that $\mathcal{G}r_a M=\bigoplus_{b\geqslant a}M_b$.
A graded $A$-module $M$ of finite length is \emph{very rigid} if its graded filtration coincides with its radical and socle filtrations, i.e., if it is rigid and
 $\rad^t M=\mathcal{G}r_{z+t}M$ for each $t\in[1,r]$.

We will need the following lemma, see \cite[Prop.~2.4.1]{BGS} and \cite[Lem.~2.13]{HM}.
\begin{lem}
\label{lem_rigidity}
Suppose that $A$ is an $\bbN$-graded $\bfk$-algebra, that $A_0$ is semisimple and $A$ is generated by $A_0$ and $A_1$. Let $M$ be any graded finite dimensional $A$-module.

{\rm(a)} The radical filtration of $M$ coincides with the grading filtration of $M$, up to shift, if $M/\rad M$ is irreducible.

{\rm(b)} The socle filtration of $M$ coincides with the grading filtration of $M$, up to shift, if $\soc M$ is irreducible.

Consequently, the module M is very rigid if $\soc M$ and $M/\rad M$ are both irreducible. \qed
\end{lem}

\medskip

\subsection{Cyclotomic Hecke and $q$-Schur algebra}
\label{subs_def-Hecke-Schur}
Let $\zeta$ be an $e$th primitive root of unity. Set $Q_k=\zeta^{s_k}$ for $k=1,\cdots,l$. Fix $d\in\bbN$.
\begin{df}
The \emph{cyclotomic Hecke algebra} $H^{\bfs}_d$ is the $\bfk$-algebra with generators $T_0,T_1,\cdots,T_{d-1}$ modulo the following defining relations:
\begin{align*}
(T_0-Q_1)\cdots(T_0-Q_l)&=0,\\
T_0T_1T_0T_1&=T_1T_0T_1T_0,\\
(T_k+1)(T_k-\zeta)&=0\quad &\text{for }& 1\leqslant k\leqslant d-1,\\
T_kT_{k+1}T_k&=T_{k+1}T_kT_{k+1}&\text{for }& 1\leqslant k\leqslant d-2,\\
T_iT_j&=T_jT_i&\text{for }&0\leqslant i\leqslant j-1\leqslant d-2.\\
\end{align*}
\end{df}

For each $k\in[1,d]$, set $L_k=q^{1-k}T_{k-1}\cdots T_1T_0T_1\cdots T_{k-1}$. For a composition $\lambda=(\lambda_1,\cdots,\lambda_k)$ of $d$, let $\frakS_\lambda$ be the Young subgroup $\frakS_{\lambda_1}\times\cdots\times\frakS_{\lambda_k}\subset\frakS_d$. Similarly, for an $l$-composition $\lambda=(\lambda^1,\cdots,\lambda^l)$ of $d$ let $\frakS_\lambda$ be the Young subgroup $\frakS_{\lambda^1}\times\cdots\times\frakS_{\lambda^l}\subset\frakS_d$. Let $t_1,\cdots,t_{d-1}\in\frakS_d$ be the transpositions $(1,2),(2,3), \cdots, (d-1,d)$. For a permutation $w\in\frakS_d$ with reduced expression $w=t_{j_1}\cdots t_{j_r}$ we set $T_w=T_{j_1}\cdots T_{j_r}$.
Suppose that $\bfa=(a_1,\cdots,a_l)$ is an $l$-tuple of integers in $[1,d]$. Set $u_\bfa^+=u_{\bfa,1}u_{\bfa,2}\cdots u_{\bfa,l}$, where $u_{\bfa,k}=\prod_{r=1}^{\bfa_k}(L_r-Q_k)$.
Let $\lambda$ be an $l$-composition of $d$. We associate with it the tuple $\bfa=(a_1,\cdots,a_l)$ such that $a_k=\sum_{r=1}^{k-1}|\lambda^r|$. Set $x_\lambda=\sum_{w\in \frakS_\lambda}T_w$ and $m_\lambda=u_\bfa^+x_\lambda$.

\begin{df}
The cyclotomic $q$-Schur algebra $S^{\bfs}_d$ is the $\bfk$-algebra
$$
\End_{H^{\bfs}_d}(\bigoplus_{\lambda\in C^l_d}m_\lambda H^{\bfs}_d).
$$
\end{df}

Let $\lambda\in\calP_d^l$. Let $\Delta^\lambda$ be the Weyl module over $S^{\bfs}_d$, see, e.g., \cite[Def.~4.12]{Mat}.
Let $L^\lambda$ be the top of $\Delta^\lambda$ and let $P^\lambda$ be its projective cover in $\mod(S_{d}^{\bfs})$.
The following is proved in \cite[Thm.~4.14]{Mat}.

\begin{lem}
The category  $\mod(S^{\bfs}_d)$ is a highest weight category with standard modules $\Delta^\lambda$, where $\lambda\in \calP^l_{d}$. \qed
\end{lem}

In \cite[Sec.~5.1]{Mat}, an idempotent $\bfe\in S^{\bfs}_d$ such that $\bfe S^{\bfs}_d\bfe\simeq H^{\bfs}_d$ is defined. Let $F\colon \mod(S_{d}^{\bfs})\to\mod(H_{d}^{\bfs}), M\mapsto \bfe M$ be the Schur functor. It is a quotient functor, see \cite[Prop.~III.2.5]{Gab} and \cite[Sec.~2.4]{HM}.
The \emph{Specht module} $S^\lambda$ is the $H_{d}^{\bfs}$-module $F(\Delta^\lambda)$. The \emph{Young module} $Y^\lambda$ is the $H_{d}^{\bfs}$-module
$F(P^\lambda)$. Set $D^\lambda=F(L^\lambda)$. The module $D^\lambda$ is irreducible or zero because $F$ is a quotient functor. Set $\calK_{d}^l=\{\lambda\in\calP_{d}^l; D^\lambda\ne\{0\}\}$.

There exists a decomposition into a direct sum of unital indecomposable $\bfk$-algebras $H^\bfs_d=\bigoplus_{\bfd\in\bbN I, |\bfd|=d}H^\bfs_\bfd$, see \cite[Thm.~A]{LM} and \cite[Prop.~3.1]{RSVV}. Let $1_\bfd$ be the unit of $H^\bfs_\bfd$.
Let $\overline 1_\bfd$ be the element of $S^\bfs_d$ given by the multiplication by $1_\bfd$. We get a decomposition of the cyclotomic $q$-Schur algebra $S^\bfs_d=\bigoplus_{|\bfd|=d}S^\bfs_\bfd$, where $S^\bfs_\bfd=\overline 1_\bfd S^\bfs_d$. Each summand $S^\bfs_\bfd$ is indecomposable, because the numbers of blocks of $S_d^\bfs$ and $H_d^\bfs$ are the same since $S_d^\bfs$, $H_d^\bfs$ have isomorphic centers by the double centralizer property, see \cite[Prop.~4.33]{R} and \cite[Cor.~5.18]{GGOR}.
For each $\bfd$, we consider the idempotent $\bfe_\bfd=\bfe \overline 1_\bfd\in S^\bfs_\bfd$. The Schur functor restricts to a quotient functor $F\colon \mod(S^\bfs_\bfd)\to \mod(H^\bfs_\bfd), M\mapsto \bfe_\bfd M$.

\begin{rk}
\label{rk_full-faith-proj}
(a) The algebra $S^\bfs_\bfd$ is quasi-hereditary with the set of standard modules $\{\Delta^\lambda;\lambda\in\calP^l_\bfd\}$, see \cite[Section 5.2]{Mat}.

(b) The functor $F$ is fully faithful on projective modules. See, for example, \cite[Prop.~4.33]{R}.
\end{rk}

For each $\lambda\in\calP^l_\bfd$, let $\nabla^\lambda\in\mod(S^\bfs_\bfd)$ be the costandard object  with  socle $L^\lambda$.

\medskip

\subsection{Cyclotomic KLR-algebra}
\label{subs_KLR}
\begin{df}
The \emph{KLR-algebra} $\tilde R_d$ is the $\bfk$-algebra with the set of generators $\psi_1,\cdots,\psi_{d-1},y_1,\cdots,y_d,e(\ui)$, where $\ui=(i_1,\cdots,i_d)\in I^d$,
modulo the following defining relations
\begin{align*}
e(\ui)e(\uj)&=\delta_{\ui,\uj}e(\ui), &\sum_{\ui\in I^d}e(\ui)&=1,\\
y_re(\ui)&=0,e(\ui)y_r,&\psi_re(\ui)&=e(t_r(\ui))\psi_r, &y_ry_s&=y_sy_r,\\
\end{align*}
\begin{align*}
\psi_ry_{r+1}e(\ui)&=(y_r\psi_r+\delta_{i_r,i_{r+1}})e(\ui), &y_{r+1}\psi_re(\ui)&=(\psi_ry_r+\delta_{i_r,i_{r+1}})e(\ui),\\
\psi_ry_s&=y_s\psi_r, &\text{if }&s\ne r,r+1,\\
\psi_r\psi_s&=\psi_s\psi_r, &\text{if }&|r-s|>1,\\
\end{align*}
$$
\psi_r^2e(\ui)=
\begin{cases}
0, &\text{if } i_r=i_{r+1},\\
(y_{r+1}-y_r)e(\ui), &\text{if } i_r=i_{r+1}-1, e\ne 2,\\
(y_r-y_{r+1})e(\ui), &\text{if } i_r=i_{r+1}+1, e\ne 2,\\
(y_r-y_{r+1})(y_{r+1}-y_r)e(\ui), &\text{if } i_r\ne i_{r+1},e=2,\\
e(\ui), &\text{\emph{otherwise}},
\end{cases}
$$
$$
\psi_r\psi_{r+1}\psi_re(\ui)=
\begin{cases}
0, &\text{if } i_r=i_{r+1},\\
(\psi_{r+1}\psi_r\psi_{r+1}+1)e(\ui), &\text{if } i_r=i_{r+2}=i_{r+1}-1, e\ne 2,\\
(\psi_{r+1}\psi_r\psi_{r+1}-1)e(\ui), &\text{if } i_r=i_{r+2}=i_{r+1}+1, e\ne 2,\\
(\psi_{r+1}\psi_r\psi_{r+1}+y_r-2y_{r+1}+y_{r+2})e(\ui), &\text{if } i_r=i_{r+2}\ne i_{r+1}, e=2,\\
\psi_{r+1}\psi_r\psi_{r+1}e(\ui), &\text{\emph{otherwise}},
\end{cases}
$$
for each $\ui,\uj\in I^d$ and each admissible $r$ and $s$. The algebra $\tilde R_d$ admits a $\bbZ$-grading such that $\deg e(\ui)=0$, $\deg y_r=2$, $\deg\psi_se(\ui)=-a_{i_s,i_{s+1}}$, for each $1\leqslant r\leqslant d$, $1\leqslant s< d$ and $\ui\in I^d$.   Here $(a_{i,j})$ is the Cartan matrix defined in Section \ref{subs_comb}.
\end{df}

\begin{df}
\label{def_KLR-cycl}
For each $i\in I$, let $c_i$ be the number of elements of $\bfs$ equal to $i$. The \emph{cyclotomic KLR-algebra} $R^{\bfs}_d$ is the quotient of the algebra $\tilde R_d$ by the ideal generated by $y_1^{c_{i_1}}e(\ui)$ for $\ui=(i_1,\cdots,i_d)\in I^d$. The algebra $R^\bfs_d$ inherits the grading from the grading of $\tilde R_d$.
\end{df}

For each $\bfd\in\bbN I$ such that $|\bfd|=d$ set $e(\bfd)=\sum_{\ui\in I^\bfd} e(\ui)\in \tilde R_d$. It is a homogeneous idempotent concentrated in degree zero. We will also denote by $e(\bfd)$ the image of $e(\bfd)$ in $R^\bfs_d$. We have the following decomposition in sums of unital $\bfk$-algebras
$\tilde R_d=\bigoplus_{|\bfd|=d}\tilde R_\bfd$, $R^\bfs_d=\bigoplus_{|\bfd|=d}R^\bfs_\bfd$, where $\tilde R_\bfd=e(\bfd)\tilde R_d$, $R^\bfs_\bfd=e(\bfd)R^\bfs_d$.

The following theorem is proved in \cite[Sec.~4.5]{BK09} and \cite[Sec.~3.2.6]{R0}.
\begin{thm}
\label{thm_isom-KLR-Hecke}
There is a $\bfk$-algebra isomorphism $R^\bfs_d\simeq H^\bfs_d$. Moreover, under this isomorphism the idempotent $1_\bfd\in H^\bfs_d$ corresponds to $e(\bfd)\in R^\bfs_d$. In particular for each $\bfd\in \bbN I$ such that $|\bfd|=d$ we get an algebra isomorphism $R^\bfs_\bfd\simeq H^\bfs_\bfd$. \qed
\end{thm}

We fix the isomorphism $R^\bfs_d\simeq H^\bfs_d$ as in \cite[Thm.~4.1]{BK11}. It yields a grading of the cyclotomic Hecke algebra $H^\bfs_d$ and on each direct factor $H^\bfs_\bfd$.

The category $\mod(R^\bfs_\bfd)$ admits a graded duality defined in the beginning of \cite[Sec.~4.5]{BK11}. We denote by $D'\colon\mod(R^\bfs_\bfd)\to\mod(R^\bfs_\bfd)$ the ungraded version of this duality and we denote by $\overline D'\colon\grmod(R^\bfs_\bfd)\to\grmod(R^\bfs_\bfd)$ the graded lift of $D'$.

\medskip

\subsection{Quiver Schur algebra\label{subs_QSchur}}
Let $\Gamma$ be the quiver as in Section \ref{subs_comb}. Let $V=\bigoplus V_i$ be an $I$-graded finite dimensional $\bbC$-vector space with dimension vector $\bfd=\sum_{i\in I}d_i\cdot i\in\mathbb NI$. Set
$
E_V=\bigoplus_{i\in I}{\rm Hom}(V_i,V_{i+1})$ and $|\bfd|=\sum_{i\in I}d_i$.
The group $G_V=\prod_{i\in I}GL(V_i)$ acts on $E_V$ by conjugation.

A \emph{vector composition} of weight $\bfd$ is a tuple $\mu=(\mu^{(1)},\cdots, \mu^{(k)})$ of nonzero elements of $\bbN I$ such that $\mu^{(1)}+\mu^{(2)}+\cdots+\mu^{(k)}=\bfd$. Let $\mathrm{VComp}(\bfd)$ be the set of vector compositions of weight $\bfd$. Set also $\mathrm{VComp}=\coprod_{\bfd\in \mathbb NI}\mathrm{VComp}(\bfd)$.
For $\mu=(\mu^{(1)},\cdots, \mu^{(k)})\in \mathrm{VComp(\bfd)}$ we denote by $F_\mu$ the variety of all flags of $I$-graded vector spaces
$
\phi=(0=V^0\subset V^1\subset\cdots\subset V^k=V)
$
in $V=\bigoplus_{i\in I}V_i$ such that the $I$-graded vector space $V^k/V^{k-1}$ has graded dimension $\mu^{(r)}$ for $r\in\{1,2,\cdots,k\}$. Let $\widetilde{F}_\mu$ be the variety of pairs $(x,\phi)\in E_V\times F_\mu$ that are compatible, i.e., we have $x(V^r)\subset V^{r-1}$ for $r\in\{0,1,\cdots,k\}$. Let $\pi_{\mu}$ be the projection $\widetilde{F}_\mu\to E_V$ such that
$
(x,\phi)\mapsto x.
$

We call an $(l+1)$-tuple $\tilde\mu=(\mu_0,\mu_1,\cdots,\mu_l)$ of elements of $\mathrm{VComp}$ a \emph{weight data}. The \emph{weight} of a weight data $\tilde\mu=(\mu_0,\mu_1,\cdots,\mu_l)$ is the sum of the weights of $\mu_0,\mu_1,\cdots,\mu_l$. Let $\mathrm{VDat}(\bfd)$ be the set of weight data of weight $\bfd$.
Let $\tilde\mu=(\mu_0,\mu_1,\cdots,\mu_l)$ be in $\mathrm{VDat}(\bfd)$. Let $\mu$ be the vector composition associated with $\tilde\mu$, i.e., $\mu$ is the tuple obtained by taking together all the components $\mu_0,\cdots,\mu_l$ of $\tilde\mu$ (ordered from $\mu_0$ to $\mu_l$).

For each $\nu\in\bbN I$ and $i\in I$ we will write $\nu(i)$ for the $I$-component of $\nu$, i.e., we have $\nu=\sum_{i\in I}\nu(i)\cdot i$.
Let $\pmb\nu=(\nu_1,\cdots,\nu_l)$ be an $l$-tuple of elements of $\mathbb NI$. Set $\nu=\nu_1+\cdots+\nu_l$.
Let $U=\bigoplus_{i\in I}U_i$ be an $I$-graded $\bbC$-vector space with $\dim U_i=\nu(i)$. Fix an $I$-graded flag $0=U^0\subset U^1\subset\cdots\subset U^l=U$ such that the $I$-graded vector space $U^k/U^{k-1}$ has graded dimension $\nu_k$ for each $k\in[1,l]$. For $i\in I, k\in[0,l]$ let $U_i^k$ be the $i$-component of $U^k$, i.e., $U^k_i=U_i\cap U^k$.

Let $\mathfrak O(\pmb\nu,\tilde\mu)$ be the set of pairs
$
((x,\phi),(\gamma_i)_{i\in I})$ in $\widetilde{F}_\mu\times \bigoplus_{i\in I}\Hom(V_i,U_i)
$
such that we have
$\gamma_i(\widetilde W_i(k))\subset U_i^k$ for each $i\in I$ and $k\in[0,l]$.
Here
$\widetilde W_i(0)\subset\widetilde W_i(1)\subset\cdots\subset\widetilde W_i(l)=V_i$
is the $I$-graded flag in $V$ whose components are among the components of the flag $\phi$ and the graded dimensions of
$$
\widetilde W_i(0),~ \widetilde W_i(1)/\widetilde W_i(0),\cdots, ~\widetilde W_i(l)/\widetilde W_i(l-1)
$$
are the weights of the components $\mu_0,\mu_1,\cdots,\mu_l$ of $\tilde \mu$ respectively.
Set $E_{V,\nu}=E_V\times \bigoplus_{i\in I}\Hom(V_i,U_i)$.
For each $\tilde\mu\in\mathrm{VDat}(\bfd)$, let $\pi_{\tilde\mu,\pmb\nu}$ be the projection $$
\pi_{\tilde\mu,\pmb\nu}\colon \mathfrak O(\pmb\nu,\tilde\mu)\to E_{V,\nu},\quad ((x,\phi),(\gamma_i)_{i\in I})\to (x,(\gamma_i)_{i\in I}).
$$
For each $\tilde\mu,\tilde\lambda\in\mathrm{VDat}(\bfd)$ set
$
Z^{\pmb\nu}_{\tilde\mu,\tilde\lambda}=\mathfrak O(\pmb\nu,\tilde\mu)\times_{E_{V,\nu}}\mathfrak O(\pmb\nu,\tilde\lambda).
$
Set also
$$
Z^{\pmb\nu}_\bfd=\coprod\limits_{\tilde\mu,\tilde\lambda\in\mathrm{VDat}(\bfd)}Z^{\pmb\nu}_{\tilde\mu,\tilde\lambda},\qquad \mathfrak O^{\pmb\nu}_{\bfd}=\coprod\limits_{\tilde\lambda\in\mathrm{VDat}(\bfd)}\mathfrak O(\pmb\nu,\tilde\lambda).
$$

\begin{df}
\label{def_QSchur}
Let $\tilde A^{\pmb\nu}_{\bfd}$ be the $\bfk$-algebra $\mathrm{H}_*^{G_V}(Z^{\pmb\nu}_\bfd)$, where $\mathrm{H}_*^{G_V}(\bullet)$ is the $G_V$-equivariant Borel-Moore homology with coefficients in $\bfk$ and the multiplication on $\tilde A^{\pmb\nu}_{\bfd}$ is given by the convolution product with respect to the inclusion $Z^{\pmb\nu}_\bfd\subset \mathfrak O^{\pmb\nu}_\bfd\times \mathfrak O^{\pmb\nu}_\bfd$, see \cite[Section 2.7]{CG}.
\end{df}

\begin{rk}
\label{rem_def-QS}
(a)
We can identify the $\bfk$-algebra $\tilde A^{\pmb\nu}_{\bfd}$ with the Yoneda extension algebra of
$\bigoplus_{\tilde\mu\in\mathrm{VDat}(\bfd)}(\pi_{\tilde\mu,\pmb\nu})_*\underline{\bfk}_{\mathfrak O(\pmb\nu,\tilde\mu)}[\dim \mathfrak O(\pmb\nu,\tilde\mu)]$
(in the $G_V$-equivariant derived category of constructible complexes on $E_{V,\nu}$), see \cite[Lem.~8.6.1,~Thm.~8.6.7]{CG}. Here $\underline\bfk$ is the constant sheaf. We equip $\tilde A^{\pmb\nu}_{\bfd}$ with the grading such that $n$-graded component is the $n$th extension group.

(b)
For each $\bfd,\bfd'\in\bbN I$ and for any vector compositions $\pmb{\nu}$, $\pmb{\nu}'$, there is a bilinear map
$
\tilde A^{\pmb\nu}_\bfd\times \tilde A^{\pmb{\nu}'}_{\bfd'}\to \tilde A^{\pmb\nu\cup\pmb\nu'}_{\bfd+\bfd'},~(a,b)\mapsto a|b,
$
where $\pmb\nu\cup\pmb\nu'$ is the concatenation of $\pmb\nu$ with $\pmb\nu'$, see \cite[Sec.~4.1]{SW}. It is associative and respects the grading, i.e., we have $\deg (a|b)=\deg a+\deg b$.
\end{rk}
Let $I^{\pmb\nu}_\bfd\subset\tilde A^{\pmb\nu}_\bfd$ be the ideal generated by the elements of the form $e_i|a$, where $i\in I$ is such that $\bfd-i\in\bbN I$, $a\in \tilde A^{\pmb\nu}_{\bfd-i}$ and $e_i$ is the unit of $\tilde A^\emptyset_{i}$. Set $A^{\pmb\nu}_\bfd=\tilde A^{\pmb\nu}_\bfd/ I^{\pmb\nu}_\bfd$. We call $A^{\pmb\nu}_\bfd$ the \emph{cyclotomic quiver Schur algebra}.
The ideal $I_\bfd^{\pmb\nu}$ is homogeneous because each element $e_i\in \tilde A^\emptyset_{i}$ is homogeneous of degree zero and the bilinear map $\tilde A^\emptyset_{i}\times \tilde A^{\pmb\nu}_{\bfd-i}\to \tilde A^{\pmb\nu}_{\bfd}$ respects the grading. Thus $A^{\pmb\nu}_\bfd$ inherits the grading from $\tilde A^{\pmb\nu}_\bfd$.

\begin{rk}
\label{rk_cycl-ideal}
In \cite[Sec.~4.2]{SW}, a diagrammatic presentation of $\tilde A^{\pmb\nu}_\bfd$ is introduced, where $\tilde A^{\pmb\nu}_\bfd$ is generated by some \emph{diagrams}. In particular, the idempotent $e_i\in \tilde A^{\emptyset}_i$ is represented by the diagram consisting of a
black line labelled by the element $i$.
The bilinear map in Remark \ref{rem_def-QS} (b) corresponds to the horizontal concatenation of diagrams. The product in $\tilde A^{\pmb\nu}_\bfd$ corresponds to the vertical concatenation of diagrams (the product of two diagrams is zero if the vertical concatenation is not possible). Hence, the ideal $I^{\pmb\nu}_\bfd\subset \tilde A^{\pmb\nu}_\bfd$ is generated by the diagrams having a black vertical line on the left labelled by an element of $I$. In particular, each diagram having a piece of a black line labelled by an element of $I$ on the left of all other lines is in $I^{\pmb\nu}_\bfd$.
\end{rk}

From now, we assume that $\pmb\nu=\bfs$.
The following is proved in \cite[Thm.~6.3]{SW}.

\begin{thm}
\label{thm_q-Sch-QSch}
The algebra $A^\bfs_\bfd$ is a graded version of the cyclotomic
$q$-Schur algebra $S^\bfs_\bfd$.
\qed
\end{thm}

\medskip

\subsection{Admissible gradings}
\label{subs_admis-grad}
In this section we equip $H^\bfs_\bfd$ with the grading induced by the isomorphism $H^\bfs_\bfd\simeq R^\bfs_\bfd$ in Section \ref{subs_KLR}. Recall the duality $D'$ and its graded lift $\overline D'$ defined in Section \ref{subs_KLR}.
We want to introduce a class of gradings on $S^\bfs_\bfd$.
\begin{df}
\label{def_adm-grad}
A $\bbZ$-grading of the algebra $S^\bfs_\bfd$ is \emph{admissible} if the Schur functor $F\colon\mod(S^\bfs_\bfd)\to \mod(H^\bfs_\bfd)$ admits a graded lift $\overline F$ and if for each $\lambda\in\calP^l_\bfd$, the modules $\Delta^\lambda$, $L^\lambda$ have graded lifts such that the following conditions hold:
\begin{itemize}
    \item[(a)] for each $\lambda,\xi\in \calP^l_\bfd$, we have $[\Delta^{\lambda}:L^\xi]_q=\delta_{\lambda\xi}+q\bbN[q]$,
    \item[(b)] there is an equivalence of categories $D\colon\mod(S^\bfs_\bfd)^{\rm op}\to\mod(S^\bfs_\bfd)$ such that
\begin{itemize}
       \item[($\rm b_1$)] $D$ admits a graded lift $\overline D\colon\grmod(S^\bfs_\bfd)^{\rm op}\to\grmod(S^\bfs_\bfd)$ that is also an equivalence of categories,
       \item[($\rm b_2$)] $D^2=\Id_{\mod(S^\bfs_\bfd)}$, $\overline D^2=\Id_{\grmod(S^\bfs_\bfd)}$,
       \item[($\rm b_3$)] $\overline D(L^\lambda)=L^\lambda$ for each $\lambda\in\calP^l_\bfd$,
       \item[($\rm b_4$)] $\overline D(M\langle 1\rangle)=\overline D(M)\langle -1\rangle$ for each $M\in\grmod(S^\bfs_\bfd)$,
       \item[($\rm b_5$)] $D'F=FD$, $\overline D'\overline F={\overline F}\,{\overline D}$.
\end{itemize}
\end{itemize}
\end{df}

From now on, we fix an admissible grading of $S^\bfs_\bfd$. For each $\lambda\in\calP^l_\bfd$ the $S^\bfs_\bfd$-modules $\Delta^\lambda$, $\nabla^\lambda$, $P^\lambda$, $L^\lambda$ are indecomposable. Thus their graded lifts are unique modulo the shift of grading if they exist, see \cite[Lem.~2.5.3]{BGS}. Fix the graded lifts of $\Delta^\lambda$ and $L^\lambda$ as in Definition \ref{def_adm-grad}. Note that by (a) each morphism $\Delta^\lambda\to L^\lambda$ is homogeneous of degree zero.
The projective cover of $L^\lambda$ in $\grmod(S^\bfs_\bfd)$ is a graded lift of $P^\lambda$. We will always consider $P^\lambda$ with this graded lift. We also get graded lifts of $D^\lambda$, $S^\lambda$ and $Y^\lambda$ obtained by taking the images by $\overline F$ of the graded lifts of $L^\lambda$, $\Delta^\lambda$ and $P^\lambda$. Moreover, the module $\overline D(\Delta^\lambda)\in\grmod(S^\bfs_\bfd)$ is a graded lift of $\nabla^\lambda$, see \cite[Prop.~2.6]{RSVV}.

\medskip

\subsection{Quiver Schur grading is admissible} The goal of this section is to prove that the graded version $A^\bfs_\bfd$ of $S^\bfs_\bfd$ has an admissible grading. We will see a finite dimensional $\bbZ$-graded $\bfk$-vector space $U=\bigoplus_{n\in\bbZ}U_n$ as a complex with
zero differentials
 $\cdots\to U_{-1}\to U_0\to U_1\to\cdots$. The dual $U^*=\Hom_\bfk(U,\bfk)$ is also $\bbZ$-graded, with $U^*_n=\Hom_{\grmod(\bfk)}(U\langle n \rangle,\bfk)$. 

Write
$
\scrL_{\tilde\mu,\bfs}=(\pi_{\tilde\mu,\bfs})_*\underline{\bfk}_{\mathfrak O(\bfs,\tilde\mu)}[\dim \mathfrak O(\bfs,\tilde\mu)].$
By Remark \ref{rem_def-QS} we have
$$
\tilde A^\bfs_\bfd=\Ext_{G_V}^*(\bigoplus_{\tilde\mu\in\mathrm{VDat}(\bfd)}\scrL_{\tilde\mu,\bfs},\bigoplus_{\tilde\mu\in\mathrm{VDat}(\bfd)}\scrL_{\tilde\mu,\bfs}).
$$

Set $s=\sum_{k=1}^ls_l\in\bbN I$. Let $\Lambda_{V,\bfs}$ be the set of isomorphism classes of simple $G_V$-equivariant perverse sheaves on $E_{V,s}$ that appear as shifts of direct summands of $\scrL_{\tilde\mu,\bfs}$ for some $\tilde\mu\in\mathrm{VDat}(\bfd)$. For each $\lambda\in\Lambda_{V,\bfs}$, fix a perverse sheaf $\IC_\lambda$ on $E_{V,s}$ in the class $\lambda$.
For each $\tilde\mu\in\mathrm{VDat}(\bfd)$, we have an isomorphism
$$
\scrL_{\tilde\mu,\bfs}\simeq\bigoplus_{\lambda\in\Lambda_{V,\bfs}}\IC_\lambda\otimes U_{\tilde\mu,\bfs,\lambda}
$$
for some $\bbZ$-graded vector spaces $U_{\tilde\mu,\bfs,\lambda}$. Note that we have a graded $\bfk$-vector space isomorphism $U_{\tilde\mu,\bfs,\lambda}\simeq(U_{\tilde\mu,\bfs,\lambda})^*$, because $\scrL_{\tilde\mu,\bfs}$ is self-dual with respect to the Verdier duality.
Set also $\scrL_\bfs=\bigoplus_{\tilde\mu\in\mathrm{VDat}(\bfd)}\scrL_{\tilde\mu,\bfs}$
and $U_{\bfs,\lambda}=\bigoplus_{\tilde\mu\in\mathrm{VDat}(\bfd)}U_{\tilde\mu,\bfs,\lambda}$.
Then we have $\scrL_\bfs=\bigoplus_{\lambda\in\Lambda_{V,\bfs}}\IC_\lambda\otimes U_{\bfs,\lambda}$.

The graded algebra
$$
^b\tilde A^\bfs_\bfd=\Ext_{G_V}^*(\bigoplus_{\lambda\in\Lambda_{V,\bfs}}\IC_\lambda,\bigoplus_{\lambda\in\Lambda_{V,\bfs}}\IC_\lambda)
$$
is Morita equivalent to $\tilde A^\bfs_\bfd$. The corresponding equivalence of (ungraded) module categories is given by the functor $Y\colon\mod(\tilde A^\bfs_\bfd)\to\mod(^b\tilde A^\bfs_\bfd)$, $M\mapsto R\otimes_{\tilde A^\bfs_\bfd}M$, where $R$ is the graded $(^b\tilde A^\bfs_\bfd,\tilde A^\bfs_\bfd)$-bimodule $$
\Ext^*_{G_V}(\bigoplus_{\lambda\in\Lambda_{V,\bfs}}\IC_\lambda\otimes U_{\bfs,\lambda},\bigoplus_{\lambda\in\Lambda_{V,\bfs}}\IC_\lambda).
$$
Let $\overline Y\colon\grmod(\tilde A^\bfs_\bfd)\to\grmod(^b\tilde A^\bfs_\bfd)$ be the obvious graded lift of $Y$.

The algebra $\tilde A^\bfs_\bfd$ admits a graded anti-involution $i\colon\tilde A^\bfs_\bfd\to ({\text{}{\tilde A}_\bfd^{\bfs}})^\mathrm{op}$ coming from the Verdier duality.

\begin{lem}
The graded anti-involution $i$ factors to the graded anti-involution $i$ of $A^\bfs_\bfd$.
\end{lem}

\begin{proof}
Recall that $\tilde A^\bfs_\bfd=\mathrm{H}_*^{G_V}(Z^\bfs_\bfd)$, where $Z^{\bfs}_\bfd=\coprod_{\tilde\mu,\tilde\lambda\in\mathrm{VDat}(\bfd)}Z^{\bfs}_{\tilde\mu,\tilde\lambda}$ and $Z^{\bfs}_{\tilde\mu,\tilde\lambda}=\mathfrak O(\bfs,\tilde\mu)\times_{E_{V,s}}\mathfrak O(\bfs,\tilde\lambda)$.
Exchanging $\mathfrak O(\bfs,\tilde\mu)$ with $\mathfrak O(\bfs,\tilde\lambda)$ and summing up over all $\tilde\lambda,\tilde\mu\in\mathrm{VDat}(\bfd)$ yields an automorphism of $Z^\bfs_\bfd$. This automorphism gives an anti-involution of $\tilde A^\bfs_\bfd$. We can see from the chain of equalities in \cite[(8.6.4)]{CG} that this anti-involution coincides with $i$.

Recall the diagrammatic presentation in Remark \ref{rk_cycl-ideal}.
The symmetry in the definition of \emph{splits} and \emph{merges} \cite[Def.~3.1]{SW} and of \emph{left} and \emph{right shifts} \cite[Def.~4.4]{SW}
implies that the anti-involution $i$ inverses each diagram in the presentation of $\tilde A^\bfs_\bfd$ above.

Now, the diagrammatic characterization of the ideal $I^\bfs_\bfd$ is Remark \ref{rk_cycl-ideal} implies that $I^\bfs_\bfd$ is preserved by $i$.
\end{proof}

In the same way as above we can define an anti-involution $i\colon{^b\tilde A^\bfs_\bfd}\to ({^b\tilde A^\bfs_\bfd})^\mathrm{op}$.
Now, for each $A\in\{\tilde A^\bfs_\bfd,A^\bfs_\bfd,{^b\tilde A^\bfs_\bfd}\}$, we have a duality $D\colon\mod(A)^{\rm op}\to \mod(A)$ such that, for each
$M\in \mod(A)$, the vector space $D(M)$ is isomorphic to $\Hom_\bfk(M,\bfk)$ and the $A$-action is twisted by $i$. Let $\overline D\colon\grmod(A)^{\rm op}\to \grmod(A)$ be the obvious graded
 lift of $D$.
 Consider the graded $(\tilde A^\bfs_\bfd,{^b\tilde A^\bfs_\bfd})$-bimodule
$$
R'=\Ext_{G_V}^*(\bigoplus_{\lambda\in\Lambda_{V,\bfs}}\IC_\lambda,\bigoplus_{\lambda\in\Lambda_{V,\bfs}}\IC_\lambda\otimes U_{\bfs,\lambda}).
$$

\begin{lem}
\label{lem_funct-isom-R-R'}
{\rm(a)} There is an $(^b\tilde A^\bfs_\bfd,\tilde A^\bfs_\bfd)$-bimodule isomorphism $R\simeq \Hom_{\tilde A^\bfs_\bfd}(R',\tilde A^\bfs_\bfd)$.

{\rm(b)} There is an isomorphism $
R\otimes_{\tilde A^\bfs_\bfd}\bullet\simeq \Hom_{\tilde A^\bfs_\bfd}(R',\bullet)
$ of functors $\mod(\tilde A^\bfs_\bfd)\to\mod(^b\tilde A^\bfs_\bfd)$.

{\rm(c)} The twist by $i$ yields an $(\tilde A^\bfs_\bfd,{^b\tilde A^\bfs_\bfd})$-bimodule structure on $R$ which is isomorphic to $R'$.

{\rm(d)} The duality $D$ on $\mod(A^\bfs_\bfd)$ satisfies the properties $\rm (b_1)$, $\rm (b_2)$, $\rm (b_4)$.
\end{lem}
\begin{proof}
Parts (a), (c) and (d) are easy. Let us concentrate on (b). Assume that $M\in \tilde A^\bfs_\bfd$. The $\tilde A^\bfs_\bfd$-module $R'$ is projective and finitely generated. Thus, we have  $\Hom_{\tilde A^\bfs_\bfd}(R',M)\simeq\Hom_{\tilde A^\bfs_\bfd}(R',\tilde A^\bfs_\bfd)\otimes_{\tilde A^\bfs_\bfd}M$. Now, part (b) follows from (a).
\end{proof}

\begin{lem}
\label{lem_YD=DY}
We have  isomorphisms of functors $YD\simeq DY$ and $\overline Y\,\overline D\simeq\overline D\,\overline Y$.
\end{lem}

\begin{proof}
For each $M\in \mod(\tilde A^\bfs_\bfd)$ we have the following chain of ${^b\tilde A^\bfs_\bfd}$-module isomorphisms
\begin{align*}
DY(M)&\simeq\Hom_\bfk(R\otimes_{\tilde A^\bfs_\bfd} M,\bfk)\\
&\simeq\Hom_{{\text{}\tilde A^\bfs_\bfd}^\mathrm{op}}(R,\Hom_\bfk(M,\bfk))\\
&\simeq\Hom_{{\text{}\tilde A^\bfs_\bfd}}(R',\Hom_\bfk(M,\bfk))\\
&\simeq\Hom(R\otimes_{\tilde A^\bfs_\bfd}\Hom_\bfk(M,\bfk))\\
&\simeq YD(M).
\end{align*}
Here, the third isomorphism is Lemma \ref{lem_funct-isom-R-R'} (c) and the fourth one is Lemma \ref{lem_funct-isom-R-R'} (b).
Thus we get $YD\simeq DY$. We can prove that $\overline Y\,\overline D\simeq\overline D\,\overline Y$ in a similar way.
\end{proof}

\begin{lem}
\label{lem_QS-gr-lift-simple}
For each $\lambda\in\calP^l_\bfd,$ the simple $S^\bfs_\bfd$-module $L^\lambda$ admits a unique graded lift
in $\grmod(A^\bfs_\bfd)$ which is stable by the duality $\overline D$.
\end{lem}

\begin{proof}
To prove the existence of the graded lift, it is enough to check that $L^\lambda$ has a $\overline D$-stable graded lift, where $L^\lambda$ is viewed as an $\tilde A^\bfs_\bfd$-module.
Further, by Lemma \ref{lem_YD=DY}, it is enough to show that the $^b\tilde A^\bfs_\bfd$-module $Y(L^\lambda)$ admits a $\overline D$-stable graded lift.
The last statement is obvious because the algebra $^b\tilde A^\bfs_\bfd$ is non negatively graded, and thus $Y(L^\lambda)$ admits a graded lift concentrated in degree zero. This graded lift is obviously $\overline D$-stable. The graded lift of $L^\lambda$ is unique up to a shift of the grading because $L^\lambda$ is indecomposable. Only one of the graded lifts of $L^\lambda$ is fixed by $\overline D$ by the property $\rm (b_4)$.
\end{proof}

Lemma \ref{lem_QS-gr-lift-simple} shows that the condition $\rm (b_3)$ holds.
From now on, let $L^\lambda$ denote also the graded $A^\bfs_\bfd$-module in Lemma \ref{lem_QS-gr-lift-simple}. Let $P^\lambda$ be the projective cover of $L^\lambda$ in $\grmod(A^\bfs_\bfd)$.
The $A^\bfs_\bfd$-module $\Delta^\lambda$ admits graded lift, see the proof of \cite[Cor.~4]{MO}. It is unique up to a shift of the grading because $\Delta^\lambda$ is indecomposable.
Let $\Delta^\lambda$ denote also the unique graded $A^\bfs_\bfd$-module which is a graded lift of $\Delta^\lambda$ such that
the obvious map $\Delta^\lambda\to L^\lambda$ is homogeneous of degree $0$.
Now, we can check the condition $\rm (a)$.

\begin{lem}
\label{lem_coef-pos}
We have $[\Delta^\lambda:L^\xi]_q\in\delta_{\lambda\xi}+q\bbN[q]$ for each $\lambda,\xi\in\calP^l_\bfd$.
\end{lem}

\begin{proof}
It is enough to prove that $[P^\lambda:L^\xi]_q\in\delta_{\lambda\xi}+q\bbN[q],$ because $\Delta^\lambda$ is a quotient of $P^\lambda$.
Let $\pi$ be the quotient morphism $\pi\colon \tilde A^\bfs_\bfd\to A^\bfs_\bfd$.
Let $\pi^*\colon\grmod(A^\bfs_\bfd)\to \grmod(\tilde A^\bfs_\bfd)$ be the pull-back by $\pi$, and
$\pi_!\colon \grmod(\tilde A^\bfs_\bfd)\to\grmod(A^\bfs_\bfd)$ be the left adjoint functor to $\pi^*$, i.e.,
 we have $\pi_!(N)=A^\bfs_\bfd\otimes_{\tilde A^\bfs_\bfd}N$ for each $N$.

 Let $\tilde P^\lambda$ be the projective cover of $\pi^*(L^\lambda)$ in $\grmod(\tilde A^\bfs_\bfd)$.
 We first claim that $\pi_!(\tilde P^\lambda)=P^\lambda$. Note that, since $\pi_!$ is left adjoint of an exact functor, it takes projectives to projectives.
Now, by adjunction, for each $\lambda,\xi\in\calP^l_\bfd$, we have
$$
\dim_q\Hom_{\tilde A^\bfs_\bfd}(\pi_!(\tilde P^\lambda),L^\xi)=\dim_q\Hom_{A^\bfs_\bfd}(\tilde P^\lambda,\pi^*(L^\xi))=\delta_{\lambda\xi}.
$$
Thus $\pi_!(\tilde P^\lambda)$ is the projective cover of $L^\lambda$ in $\grmod(A^\bfs_\bfd)$. This proves the claim.

The proof of Lemma \ref{lem_grBGG1} implies that
$[P^\lambda:L^\xi]_q=\dim_q\Hom_{A^\bfs_\bfd}(P^\xi,P^\lambda)$ and
$[\tilde P^\lambda:\pi^*(L^\xi)]_q=\dim_q\Hom_{\tilde A^\bfs_\bfd}(\tilde P^\xi,\tilde P^\lambda)$.
We claim that the functor $\pi_!$ yields a surjection
$\Hom_{\tilde A^\bfs_\bfd}(\tilde P^\xi,\tilde P^\lambda)\to\Hom_{A^\bfs_\bfd}(P^\xi,P^\lambda)$.
Indeed, since the right hand side is equal to $\Hom_{\tilde A^\bfs_\bfd}(\tilde P^\xi,\pi^*\pi_!(\tilde P^\lambda))$ by adjunction and since
$\Hom_{\tilde A^\bfs_\bfd}(\tilde P^\xi,\bullet)$ is exact, the claim follows from the surjectivity of the unit morphism $\tilde P^\lambda\to \pi^*\pi_!(\tilde P^\lambda)$.

Now, it is enough to prove that
 $[\tilde P^\lambda:\pi^*(L^\xi)]_q\in \delta_{\lambda\xi}+q\bbN[q]$, or, equivalently, that $[\overline Y(\tilde P^\lambda):\overline Y(\pi^*(L^\xi))]_q\in \delta_{\lambda\xi}+q\bbN[q]$. We have $(^b\tilde A^\bfs_\bfd)_n=0$ if $n<0$ and $(^b\tilde A^\bfs_\bfd)_0$ is semisimple and basic.
 The module $\overline Y(\pi^*(L^\lambda))$ has dimension $1$ and it is concentrated in degree zero.
 Thus we get $\overline Y(\tilde P^\lambda)_n$ if $n<0$ and $\dim Y(\tilde P^\lambda)_0=1$. This implies the statement.
\end{proof}

Let $\mathrm{VDat}^0(\bfd)$ be the set of elements $\tilde\mu=(\mu_0,\cdots,\mu_l)\in\mathrm{VDat}(\bfd)$ such that $\mu_0,\cdots,\mu_{l-1}$ are empty and each component of $\mu_l$ is in $I$ (viewed as a subset of $\bbN I$). Let $\mathrm{VDat}^1(\bfd)$ be the set of elements $\tilde\mu=(\mu_0,\cdots,\mu_l)\in\mathrm{VDat}(\bfd)$ such that each component of $\mu_0,\cdots,\mu_{l-1},\mu_l$ is in $I$ (viewed as a subset of $\bbN I$).

The following theorem is proved in \cite{Rouq2011} and \cite{VV}.
\begin{thm}
\label{thm_KLR-geom}
There is a graded $\bfk$-algebra isomorphism
$$
\tilde R_\bfd\simeq \Ext_{G_V}^*(\bigoplus_{\tilde\mu\in\mathrm{VDat}^0(\bfd)}\scrL_{\tilde\mu,\bfs},\bigoplus_{\tilde\mu\in\mathrm{VDat}^0(\bfd)}\scrL_{\tilde\mu,\bfs}).
$$
\qed
\end{thm}

Let $\tilde \bfe_\bfd$, $\overline \bfe_\bfd$ be the idempotent of $\tilde A^\bfs_\bfd$ concentrated in degree $0$ given by:
$$
\tilde\bfe_\bfd=\sum_{\tilde\mu\in\mathrm{VDat}^0(\bfd)}\Id_{\scrL_{\tilde\mu,\bfs}},\quad \overline\bfe_\bfd=\sum_{\tilde\mu\in\mathrm{VDat}^1(\bfd)}\Id_{\scrL_{\tilde\mu,\bfs}}.
$$
Let $\bfe_\bfd'$ be the image of $\tilde \bfe_\bfd$ in $A^\bfs_\bfd$ by the canonical map. Recall the idempotent $\bfe_\bfd$ in Section \ref{subs_def-Hecke-Schur}.
Theorem \ref{thm_KLR-geom} implies that there is a graded algebra isomorphism $\tilde \bfe_\bfd \tilde A^\bfs_\bfd \tilde \bfe_\bfd\simeq \tilde R_\bfd$. Recall the definition of the $c_i$ in Definition \ref{def_KLR-cycl}.
\begin{lem}
\label{lem_idemp-Schur-gr}
{\rm(a)} Under the identification $S^\bfs_\bfd\simeq A^\bfs_\bfd$, we have $\bfe_\bfd=\bfe_\bfd'$.

{\rm(b)} The ideal $\tilde \bfe_\bfd I^\bfs_\bfd \tilde \bfe_\bfd$ of $\tilde R_\bfd$ coincides with the ideal generated by $y_1^{c_{i_1}}e(\ui)$ for $\ui=(i_1,\cdots,i_d)\in I^\bfd$. In particular, we have a graded $\bfk$-algebra isomorphism $\bfe_\bfd A^\bfs_\bfd \bfe_\bfd\simeq R^\bfs_\bfd$.
\end{lem}
\begin{proof}
Part (a) follows from the proof of \cite[Thm.~6.3]{SW}.
Now, we prove part (b). By \cite[Prop.~4.9]{SW}, the algebra $\overline\bfe_\bfd A^\bfs_\bfd\overline\bfe_\bfd$ is isomorphic to
the \emph{tensor product algebra} with parameters $\bfs,\bfd$ which is defined in \cite[Sec. 2.1]{Web}.
Now, we identify $y_1^{c_{i_1}}\bfe(\ui)$ with an element of $\overline\bfe_\bfd A^\bfs_\bfd\overline\bfe_\bfd$ under the isomorphism
$\tilde R_\bfd\simeq \tilde \bfe_\bfd \tilde A^\bfs_\bfd \tilde \bfe_\bfd$ above.
The diagrammatic presentation mentioned in Remark  \ref{rk_cycl-ideal}
identifies $y_1^{c_{i_1}}\bfe(\ui)$ with a diagram which
starts from the left from $l$ vertical red lines with labels $s_1,\cdots,s_l$ followed by a black vertical
line labelled by $i_1$ and containing ${c_{i_1}}$ dots. Hence, the relations \cite[(2.2)]{Web} imply that
$y_1^{c_{i_1}}\bfe(\ui)\in \tilde \bfe_\bfd I^\bfs_\bfd \tilde \bfe_\bfd$, see also Remark \ref{rk_cycl-ideal}.
So, it suffices to show that both ideals in the lemma have the same codimension in $\tilde R_\bfd$ or, equivalently, that
$\dim R^\bfs_\bfd=\dim \bfe_\bfd A^\bfs_\bfd \bfe_\bfd$. This is true because the algebra $\bfe_\bfd A^\bfs_\bfd \bfe_\bfd$ is
isomorphic to the cyclotomic Hecke algebra $H^\bfs_\bfd$, see Section \ref{subs_def-Hecke-Schur}. See also Theorem \ref{thm_isom-KLR-Hecke}.
\end{proof}

Lemma \ref{lem_idemp-Schur-gr} implies that the functor $\overline F\colon\grmod(A^\bfs_\bfd)\to \grmod(R^\bfs_\bfd),~M\mapsto \bfe_\bfd M$ is a graded lift of the Schur functor $F$. Let $i'\colon \tilde R_\bfd\to (\tilde R_\bfd)^{\rm op}$ be the graded anti-involution coming from the Verdier duality, see Theorem \ref{thm_KLR-geom}. The same argument as for $i$ shows that $i'$ factors to a graded anti-involution $i'$ of $R^\bfs_\bfd$. It is easy to see that the duality $D'$ coincides with the duality functor $\mod(R^\bfs_\bfd)^{\rm op}\to\mod(R^\bfs_\bfd),~M\mapsto \Hom_\bfk(M,\bfk)$, where the $R^\bfs_\bfd$-action on $\Hom_\bfk(M,\bfk)$ is twisted by $i'$. In the same way, the graded lift $\overline D'$ of $D'$ coincides with the duality given by $\grmod(R^\bfs_\bfd)^{\rm op}\to\grmod(R^\bfs_\bfd),~M\mapsto \Hom_\bfk(M,\bfk)$. Hence, the condition $\rm (b_5)$ holds.

\medskip

\subsection{Lie algebras}
\label{subs_Lie-alg}
Let $N$ be an integer $>0$. Consider the Lie algebra $\frakg=gl_N(\bbC)$. Let $\frakh\subset\frakb\subset\frakg$ be its usual Cartan and Borel subalgebras respectively.

Let $\widehat\frakg$ be the affine Kac-Moody Lie algebra $\widehat\frakg=\frakg\otimes \bbC[t,t^{-1}]\oplus\bbC\mathbf{1}\oplus\bbC\partial$. Consider the Lie subalgebras in $\widehat{\frakg}$ given by
$$
\widehat{\frakh}=\bbC\partial\oplus\frakh\oplus\bbC\mathbf{1},\quad \widehat{\frakb}=\frakb\oplus\frakg\otimes t\bbC[t]\oplus\bbC\partial\oplus\bbC\mathbf{1}.
$$


We identify $\frakh^*$ with the vector subspace of elements of $\widehat\frakh^*$ that are zero on $\mathbf{1}$ and $\partial$. Let $\widehat\Pi$ be the set of roots of $\widehat\frakg$. Denote by $\widehat\Phi=\{\alpha_k; k\in[0,N]\}$ the subset of simple roots. For each $\alpha\in\widehat\Pi$, let $\alpha^\vee\in \widehat\frakh$ be the affine coroot associated with $\alpha$.
Let $P=\bbZ\Lambda_1\oplus\cdots\bbZ\Lambda_N\subset \frakh^*$ be the weight lattice of $\frakg$.  Let $\Lambda_0$ and $\delta$ be the elements of $\widehat\frakh^*$ defined by
$
\delta(\partial)=\Lambda_0(\mathbf{1})=1
$
and
$\delta(\frakh\oplus \bbC\mathbf{1})=\Lambda_0(\frakh\oplus\bbC\partial)=0.
$
Let $\widehat P=\bbZ\Lambda_0\oplus P\oplus\bbZ\delta$ be the weight lattice of $\widehat\frakg$. 

There exists a $\bbZ$-bilinear form $(,)\colon \widehat P\times \widehat P\to \bbQ$ such that $\lambda(\alpha_i^\vee)=2(\lambda,\alpha_i)/(\alpha_i,\alpha_i)$.
The weight $\lambda\in \widehat P$ is \emph{dominant} (resp. \emph{antidominant}) if  $\lambda(\alpha_k^\vee)\geqslant 0$ for each $k\in[0,N]$ (resp. $\lambda(\alpha_k^\vee)\leqslant 0$ for each $k\in[0,N]$).

\medskip

\subsection{Category $\bfA$}
\label{subs_catA}
Let $N$ be as above. Fix an $l$-tuple of positive integers $\bfm=(m_1,\cdots, m_l)$ such that $\sum_{t=1}^lm_t=N$. Let $\widehat\frakp_{\bfm}$ be the parabolic subalgebra $\widehat\frakb\subset\widehat\frakp_\bfm\subset\widehat\frakg$ of parabolic type $\bfm$.
Set
$$
\widehat \Phi_\bfm=\widehat\Phi\backslash \{\alpha_k; k=0,m_1,m_1+m_2,\cdots,m_1+m_2+\cdots+m_l\}.
$$
The weight $\lambda\in \widehat P$ is $\bfm$-\emph{dominant} (resp. $\bfm$-\emph{antidominant}) if $\lambda(\alpha_k^\vee)\geqslant 0$ for each $\alpha_k\in \widehat\Phi_\bfm$ (resp. $\lambda(\alpha_k^\vee)\leqslant 0$ for each $\alpha_k\in \widehat\Phi_\bfm$). Let $\widehat P^\bfm\subset \widehat P$ be the set of $\bfm$-dominant weights of $ \widehat\frakg$. Set $P^\bfm=P\cap \widehat P^\bfm$.

We will identify the element $\sum_{k=1}^Na_k\Lambda_k$, $a_k\in\bbZ$, of the weight lattice $P$ with the $N$-tuple $(a_1,a_2,\cdots,a_N)$.
Let $\rho,\rho_\bfm\in P$ be given by
$$
\rho=(0,-1,-2,\cdots,-(N-1)),
$$
$$ \rho_\bfm=(m_1,m_1-1,\cdots,1,m_2,m_2-1,\cdots,1,\cdots,m_l,m_l-1,\cdots,1).
$$
Set also $\widehat\rho=\rho+N\Lambda_0\in\widehat P$.
For each $\lambda\in P^\bfm$, consider the following element $\widehat\lambda=(z_\lambda,\lambda,-e-N)\in\widehat P^\bfm$,
where $z_\lambda=(\lambda,2\rho+\lambda)/2e$.

\begin{df}
Let $\calO^\bfm$ be the category of the finitely generated $U(\widehat\frakg)$-modules  which  are semisimple $\widehat\frakh$-modules such that the $\widehat\frakp_\bfm$-action is locally finite and the highest weights of each subquotient are of the form $\widehat\lambda$ for some $\lambda$ in $P^\bfm$. Then $\calO^\bfm$ is an abelian category.
\end{df}
For each $\lambda\in P^\bfm$, let $M^{\bfm}(\lambda)$ be the parabolic Verma module in $\calO^\bfm$ with highest weight $\widehat\lambda$.
Let $\calP^\bfm$ be the set of $l$-partitions $\lambda=(\lambda^1,\cdots,\lambda^l)$ such that $Y(\lambda^a)$ contains at most $m_a$ boxes for each $a\in[1,l]$.

We identify an $l$-partition $\lambda$ as above with the weight
$$
(\lambda^1_1,\cdots,\lambda^1_{l(\lambda^1)},0^{m_1-l(\lambda^1)},\lambda^2_1,\cdots,\lambda^2_{l(\lambda^2)},0^{m_2-l(\lambda^2)},\cdots,\lambda^l_1,\cdots,\lambda^l_{l(\lambda^l)},0^{m_l-l(\lambda^l)}),
$$
where $0^a$ is the row of $a$ zeros for each $a\in\bbN$.
Consider the map $\omega\colon\calP^\bfm\to P^\bfm$ such that $\omega(\lambda)=\lambda-\rho+\rho_\bfm$.

Assume that $\lambda\in\calP^\bfm$. Set ${\Delta}^*_\calO(\lambda)=M(\omega(\lambda))\in \calO^\bfm$.
Let $\lambda^*$ be the $l$-partition such that $Y(\lambda^*)=(Y(\lambda^l)^t,Y(\lambda^{l-1})^t,\cdots,Y(\lambda^1)^t)$, where $(\bullet)^t$ means the transposed Young diagram.
We abbreviate $\Delta_\calO^\lambda=\Delta^*_\calO(\lambda^*)$. (This notation is motivated by Theorem \ref{thm_equiv-Schur-O-ungr}.) For each $\lambda\in \calP^\bfm$ let $L_\calO^\lambda$ be the top of $\Delta^\lambda_\calO$.
\begin{df}
Let $\bfA^\bfm$ be the Serre subcategory of $\calO^\bfm$ generated by the set
$
\{{L}_\calO^\lambda;\lambda\in\calP^\bfm\}.
$
\end{df}
The category $\bfA^\bfm$ is a highest weight category with the set of standard modules $\{{\Delta}_\calO^\lambda$; $\lambda\in\calP^\bfm\}$, see \cite[Sec.~5.5]{RSVV}.
Fix $\bfd=\sum_{i\in I}d_i\cdot i\in\bbN I$ and set $\calP^\bfm_\bfd=\calP^\bfm\cap\calP_\bfd^l$.
\begin{df}
Let $\bfA^\bfm_\bfd$ be the Serre subcategory of $\bfA^\bfm$ generated by the set $\{{L}_\calO^\lambda; \lambda\in\calP^\bfm_\bfd\}$.
\end{df}
The category $\bfA^\bfm_\bfd$ is a highest weight category with the set of standard modules $\{{\Delta}_\calO^\lambda; \lambda\in\calP^\bfm_\bfd\}$, see \cite[Sec.~5.5]{RSVV}.
For each $\lambda\in\calP^\bfm_\bfd$, let $\nabla_\calO^\lambda$ be the costandard object in $\bfA^\bfm_\bfd$ that has simple socle $L_\calO^\lambda$ and let $P_{\calO}^\lambda$ be the projective cover of $L_\calO^\lambda$ in $\bfA^\bfm_\bfd$.

\medskip

\subsection{The Koszul grading}
\label{subs_categ-equiv}
Choose an $l$-tuple $\bfm=(m_1,\cdots,m_l)$ of positive integers such that $m_p\equiv -s_{l+1-p}$ mod $e$ and $m_p\geqslant|\bfd|$ for each $p\in[1,l]$.
Set $N=m_1+\cdots+m_l$. We say that $\bfm$ is \emph{dominant} (resp. \emph{antidominant}) if $m_1\geqslant m_2\geqslant\cdots\geqslant m_l$ (resp. $m_1\leqslant m_2\leqslant\cdots\leqslant m_l$).

For simplicity, we abbreviate $\bfA=\bfA^\bfm_\bfd$.
Note that $\calP^l_{\bfd}=\calP^\bfm_{\bfd}$.
Set $P_\calO=\bigoplus_{\lambda\in \calP^l_\bfd} P_\calO^\lambda$ and $S^\calO=\End_{\bfA}(P_\calO)^{\rm op}$. The functor $\Hom_{\bfA}(P_\calO,\bullet)\colon \bfA\to \mod(S^\calO)$ is an equivalence of categories.
We have the following theorem, see \cite[Thm.~6.4]{SVV2}.
\begin{thm}
\label{thm_Kos-gr-on-A}

The $\bfk$-algebra $S^\calO$ admits a Koszul grading. \qed
\end{thm}

Here and for the rest of the paper, the grading of $S^\calO$ is the Koszul grading from Theorem \ref{thm_Kos-gr-on-A}. Set $\overline \bfA=\grmod(S^\calO)$. Let $\lambda\in\calP^l_\bfd$. The $S^\calO$-modules $\Delta_\calO^\lambda$, $\nabla_\calO^\lambda$, $P_\calO^\lambda$, $L_\calO^\lambda$ are indecomposable. Thus their graded lifts are unique modulo a shift if they exist, see \cite[Lem.~2.5.3]{BGS}. We can consider each  $L_\calO^\lambda$ as a graded $S^\calO$-module with the grading concentrated in degree zero. The projective cover of $L_\calO^\lambda$ in $\grmod(S^\calO)$ is a graded lift of $P_\calO^\lambda$. From now we will always consider $P_\calO^\lambda$ as a $\bbZ$-graded  $S^\calO$-module with this grading. By \cite[Cor.~4]{MO}, there are graded lifts of $\Delta_\calO^\lambda$ and $\nabla_\calO^\lambda$. Fix them such that the morphisms $\Delta_\calO^\lambda\to L_\calO^\lambda$ and $L_\calO^\lambda\to \nabla_\calO^\lambda$ are homogeneous of degree $0$.

The following is proved in \cite[Thm.~7.9]{RSVV}.
\begin{thm}
\label{thm_equiv-Schur-O-ungr}
Assume that $\bfm$ is dominant or antidominant. There is an equivalence of categories $E_\mathrm{S}\colon\mod(S_{\bfd}^\bfs)\to\bfA$ that takes $L^\lambda$ to $L_\calO^\lambda$ and $\Delta^\lambda$ to  $\Delta_\calO^\lambda$ for each $\lambda\in\calP^l_\bfd$. \qed
\end{thm}

\begin{rk}
\label{rk_dualBGG-A!}
Since $\bfA$ is a highest weight subcategory of $\calO^\bfm$, the BGG duality yields a duality on $\bfA$ which fixes the simple modules. Similarly, since $\bfA^!$ is a quotient of an affine parabolic category $\calO$ (at positive level), it admits also a duality. See \cite{SVV2} for details.
\end{rk}

\medskip

\subsection{Analogue of the cyclotomic KLR-algebra}
\label{subs_KLR-O}

Set $\calK_{\bfd}^l=\calK^l_{|\bfd|}\cap\calP^l_\bfd$ where $\calK^l_d$ is as in Section \ref{subs_def-Hecke-Schur}.
Now, we define an analogue of the cyclotomic KLR-algebra. Consider the idempotent of $S^\calO$ given by $\bfe^\calO=\bigoplus_{\lambda\in\calK_{\bfd}^l}\Id_{P_\calO^\lambda}$. It is homogeneous of degree zero. Set $R^\calO=\bfe^\calO S^\calO \bfe^\calO$. Then, $R^\calO$ is a $\bbZ$-graded $\bfk$-algebra.
We have the functor
$$
F^\calO\colon \mod(S^\calO)\to\mod(R^\calO),\quad M\mapsto \bfe^\calO M.
$$
Let $\overline F^\calO\colon\grmod(S^\calO)\to\grmod(R^\calO)$ be the obvious graded lift of $F^\calO$.
Consider the modules in $\mod(R^\calO)$ given by $D_\calO^\lambda=F^\calO(L_\calO^\lambda)$, $S_\calO^\lambda=F^\calO(\Delta_\calO^\lambda)$ and $Y_\calO^\lambda=F^\calO(P_\calO^\lambda)$. They admit the graded lifts given by  $\overline F^\calO(L_\calO^\lambda)$, $\overline F^\calO(\Delta_\calO^\lambda)$, $\overline F^\calO(P_\calO^\lambda)$.
The functor $F^\calO$ is a quotient functor, see \cite[Prop.~III.2.5]{Gab}, \cite[Sec.~2.4]{HM}. Thus the module $D_\calO^\lambda$ is either zero or irreducible.

By \cite[Thm.~5.34]{RSVV}, there is a projective module $T\in\bfA$ such that we have an ungraded algebra isomorphism $\End_\bfA(T)^\mathrm{op}\simeq R_\bfd^{\bfs}$. Consider the functor $F'\colon \bfA\to \mod(R_\bfd^{\bfs})$ given by $M\mapsto \Hom_{\bfA}(T,M)$. By \cite[Thm.~7.9]{RSVV} we have $F'\circ E_S=F$, where $E_S$ is as in Theorem \ref{thm_equiv-Schur-O-ungr}. In particular we get $F'(L_\calO^\lambda)=D^\lambda$ for each $\lambda\in \calP^l_\bfd$. This implies that, for each $\lambda\in \calP_\bfd^l$, the projective module $P_\calO^\lambda$ is a direct factor of $T$ if and only if $\lambda\in \calK_\bfd^l$, because $\dim\Hom_{S^\calO}(P_\calO^\lambda,L_\calO^\mu)=\delta_{\lambda,\mu}$. Hence, there is an equivalence of categories
$
E_{\mathrm{H}}\colon \mod(R_{\bfd}^\bfs)\to \mod(R^\calO)
$
such that
the following diagram of functors is commutative
$$
\begin{CD}
\mod(S_{\bfd}^\bfs) @>{E_S}>> \mod(S^\calO)\\
@VFVV                                @V{F^\calO}VV\\
\mod(R_{\bfd}^\bfs) @>{E_H}>> \mod(R^\calO).\\
\end{CD}
$$
Moreover, we have
$$
E_{\mathrm{H}}(S^\lambda)=S_\calO^\lambda, \quad E_{\mathrm{H}}(D^\lambda)=D_\calO^\lambda, \quad E_{\mathrm{H}}(Y^\lambda)=Y_\calO^\lambda, \quad \forall \lambda\in\calP^l_\bfd.
$$

\begin{rk}
Remark \ref{rk_full-faith-proj} and the commutative diagram above imply that
the functor $F^\calO$ is fully faithful on projective modules.
\end{rk}

\medskip

\subsection{Graded BGG reciprocity}
\label{subs_cond-BGG-verify}
We want to show that graded categories $\grmod(S^\bfs_\bfd)$ and $\overline\bfA$ satisfy the conditions
$(\operatorname{a}_1)$, $(\operatorname{a}_2)$ from Section \ref{subs_graded-hw-cat}.

The condition ($\operatorname{a}_1$) for $\grmod(S^\bfs_\bfd)$ and $\overline\bfA$ is already mentioned in Sections \ref{subs_admis-grad} and \ref{subs_categ-equiv}. The condition ($\operatorname{a}_2$) for $\grmod(S^\bfs_\bfd)$ follows from the
condition (c) in the definition of an admissible grading.

We must check the condition ($\operatorname{a}_2$) for $\overline\bfA$. We construct the duality functor $D$ on $\bfA$ in the same way as in \cite[Sec.~3.11]{BGS}. More precisely, for each $\lambda\in\calP^l_\bfd$, let $L^\calO_\lambda$ be the simple object of $\bfA^!$ such that $K(P_\calO^\lambda)=L^\calO_\lambda$, where $K$ is the equivalence in Theorem \ref{thm_Koszul-duality}.
Since $\bfA\simeq \mod(S^\calO)$ and $S^\calO$ is a basic algebra, by Theorem \ref{thm_Koszul-duality} (a) we have a graded algebra isomorphism $S^\calO\simeq \bigoplus_{\lambda,\mu\in\calP^l_\bfd}\Ext^*_{\bfA^!}(L^\calO_\lambda,L^\calO_\mu)^{\rm op}$. We deduce that there is a graded $\bfk$-algebra anti-involution $i\colon S^\calO\to (S^\calO)^{\rm op}$, because the duality on $\bfA^!$ from Remark \ref{rk_dualBGG-A!} gives an isomorphism $\bigoplus_{\lambda,\mu\in\calP^l_\bfd}\Ext^*_{\bfA^!}(L^\calO_\lambda,L^\calO_\mu)^{\rm op}=\bigoplus_{\lambda,\mu\in\calP^l_\bfd}\Ext^*_{\bfA^!}(L^\calO_\mu,L^\calO_\lambda)$. Hence, we have the duality $D\colon\mod(S^\calO)^{\rm op}\to\mod(S^\calO)$ such that $M\mapsto \Hom_\bfk(M,\bfk)$, where the $S^\calO$ action on $\Hom_\bfk(M,\bfk)$ is twisted by $i$.
Thus part ($\operatorname{a}_2$) holds for $\overline\bfA$.

\medskip

\subsection{Decomposition numbers}
\label{subs_dec-num}
Recall the graded multiplicities $[\bullet:\bullet]_q$ and $(\bullet:\bullet)_q$ defined in Section \ref{subs_graded-hw-cat}.
\begin{df}
\label{df_grad-mult}
For each $\lambda,\xi\in\calP_{\bfd}^l$ we have $d_{\lambda\xi}(q)=[\Delta^\lambda:L^\xi]_q$, $c_{\lambda\xi}(q)=[P^\lambda:L^\xi]_q$, $d^\calO_{\lambda\xi}(q)=[\Delta_\calO^\lambda:L_\calO^\xi]_q$ and $c^\calO_{\lambda\xi}(q)=[P^\lambda_\calO:L^\xi_\calO]_q$.
\end{df}
\begin{lem}
\label{lem_multKLR=Sch}
For each $\lambda,\xi\in\calP^l_\bfd$ we have
$
[\Delta^\lambda:L^\xi]_q=(P^\xi:\Delta^\lambda)_{q}$ and $[\Delta_\calO^\lambda:L_\calO^\xi]_q=(P_\calO^\xi:\Delta_\calO^\lambda)_{q}.
$
\end{lem}
\begin{proof}
The statement follows from Corollary \ref{coro_grBGG2} and Section \ref{subs_cond-BGG-verify}.
\end{proof}
\begin{coro}
\label{coro_mat-sym}
We have $c(q)=d(q)^td(q)$ and $c^\calO(q)=d^\calO(q)^td^\calO(q).$ \qed
\end{coro}

\begin{rk}
\label{rem_mult-Sch-KLR}
The functors $F$ and $F^\calO$ are quotient functors. Thus, we have
$
d_{\lambda\xi}(q)=[S^\lambda:D^\xi]_q$ and $d^\calO_{\lambda\xi}(q)=[S_\calO^\lambda:D_\calO^\xi]_q$
for each $\lambda\in \calP_{\bfd}^l, \xi\in\calK_{\bfd}^l.
$
\end{rk}

Recall that the graded lifts of $D^\lambda$,
$S^\lambda$, $Y^\lambda$ are the graded modules $\overline F(L^\lambda)$, $\overline F(\Delta^\lambda)$, $\overline F(P^\lambda)$. Another definition of the graded lifts of these modules is given in \cite[Thm.~4.11,~Sec.~4.10,~5.2]{BK11}. Denote the graded lifts from \cite{BK11} by $D'^\lambda$,
$S'^\lambda$, $Y'^\lambda$ respectively. We want to compare these two choices of graded lifts.

\begin{lem}
\label{lem_lifts-D=D'}
Assume that $\lambda\in\calK^l_\bfd$. Then we have the graded $R^\bfs_\bfd$-module isomorphisms $D^\lambda\simeq D'^\lambda$, $S^\lambda\simeq S'^\lambda$
and $Y^\lambda\simeq Y'^\lambda$.
\end{lem}
\begin{proof}
The ungraded $R^\bfs_\bfd$-modules $D^\lambda$, $S^\lambda$, $Y^\lambda$ are indecomposable because each of them has simple top $D^\lambda$, see \cite[Prop.~2.2~(3)]{cell}. Thus their graded lifts are unique up to a shift of the grading. The graded lifts are normalised such that each morphism $Y'^\lambda\to D'^\lambda$, $S'^\lambda\to D'^\lambda$, $Y^\lambda\to D^\lambda$, $S^\lambda\to D^\lambda$ is homogeneous of degree zero. Hence, it suffice to show that there is a graded $R^\bfs_\bfd$-module isomorphism $D^\lambda\simeq D'^\lambda$. The graded lift $D'^\lambda$ of $D^\lambda$ is caracterized by stability by the duality $\overline D'$, see \cite[Thm.~4.11]{BK11}. On the other hand we have
$$
\overline D'(D^\lambda)=\overline D'\overline F(L^\lambda)=\overline F\,\overline D(L^\lambda)=\overline F(L^\lambda)=D^\lambda.
$$
Here, the second equality is the condition $\rm (b_5)$ in Definition \ref{def_adm-grad} and the third equlity is the condition $\rm (b_3)$ in Definition \ref{def_adm-grad}.
\end{proof}

The level $l$ Fock space $F^\bfs$ associated with the parameter $\bfs=(s_1,\cdots,s_l)$ is a $\bbQ(q)$-vector space with basis $\{M_\lambda; \lambda\in\calP^l\}$ called \emph{standard basis}. It is equipped with a representation of $U_q(\widehat{sl_e})$ defined as in \cite[(3.22),~(3.23)]{BK11}. Let $V^\bfs$ be the simple $U_q(\widehat{sl_e})$-module with highest weight $\Lambda_{s_1}+\cdots+\Lambda_{s_l}$. Consider the surjection $F^\bfs\to V^\bfs$ as in \cite[(3.28)]{BK11}. Let $M'_\lambda$ be the image of $M_\lambda$ in $V^\bfs$.

For each $\lambda\in\calP^l_\bfd$ and $\xi\in\calK^l_\bfd$ set $d'_{\lambda\xi}(q)=[S'^\lambda:D'^\xi]_q$ and $c'_{\lambda\xi}=[Y'^\lambda:D'^\xi]_q$.

\begin{lem}
\label{lem_mult-O-Schur-coin}
Assume that $\lambda\in\calP_{\bfd}^l$ and $\xi\in\calK_{\bfd}^l$. Then we have $d'_{\lambda\xi}(q)=d_{\lambda\xi}^\calO(q)$.
\end{lem}
\begin{proof}
By \cite[(3.35),~Cor.~5.15]{BK11} the elements $d'_{\lambda\xi}(q)\in\bbN[q,q^{-1}]$ are the coefficients of the decomposition of $M'_\lambda$ in the canonical basis of $V^\bfs$. Further, \cite[Thm.~3.26 (i)]{U} gives an explicit formula for these coefficients in terms of Kazhdan-Lusztig polynomials. Finally, we compute the decomposition numbers in the parabolic category $\calO$ with respect to the Koszul grading in Appendix \ref{app_mult}. Now, the statement follows by comparing Lemma \ref{lem_inv-sing-par-neg} with \cite[Thm.~3.26 (i)]{U}.
\end{proof}

\begin{rk}
 (a) The element $q$ in \cite{BK11} corresponds to $q^{-1}$ in \cite{U}. In particular, the positive canonical basis in \cite{BK11} corresponds to the negative canonical basis in \cite{U}.

 (b) The category concidered in Appendix \ref{app_mult} is larger then the category $\bfA$. But the inclusion of $\bfA$ in this category respects the grading by \cite[Lem.~2.2]{SVV2} and \cite[Prop.~A3.3]{Don}.
\end{rk}


\begin{coro}
\label{coro_c=cO}
Assume that $\lambda,\xi\in\calK_{\bfd}^l$. Then we have $c_{\lambda\xi}(q)=c_{\lambda\xi}^\calO(q)$.
\end{coro}

\begin{proof}
By \cite[(3.45),~Thm.~4.18,~Thm.~5.14]{BK11}, for each $\lambda\in\calK^l_\bfd$, we have $[Y'^\lambda]=\sum_{\xi\in\calP^l_\bfd}d'_{\xi\lambda}(q)[S'^\xi]$ in $[\grmod(R^\bfs_\bfd)]$. Thus, for each $\lambda,\xi\in\calK^l_\bfd$, we have $c'_{\lambda\xi}=\sum_{\mu\in\calP^l_\bfd}d'_{\lambda\mu}d'_{\mu\xi}$. On the other hand, Corollary \ref{coro_mat-sym} implies that $c^\calO_{\lambda\xi}=\sum_{\mu\in\calP^l_\bfd}d^\calO_{\lambda\mu}d^\calO_{\mu\xi}$. Hence, by Lemma \ref{lem_mult-O-Schur-coin}, we have $c'_{\lambda\xi}=c^\calO_{\lambda\xi}$. This implies the statement, because Lemma \ref{lem_multKLR=Sch} and Lemma \ref{lem_lifts-D=D'} yield $c_{\lambda\xi}=c'_{\lambda\xi}$.
\end{proof}

\begin{lem}
\label{lem_positivity-coef}
For each $\lambda,\xi\in\calP^l_\bfd$ we have $c_{\lambda\mu}^\calO(q)\in\bbN[q]$.
\end{lem}
\begin{proof}
The $S^\calO$-module $L_\calO^\xi$ is concentrated in degree zero. The top of the $S^\calO$-module $P_\calO^\lambda$ is isomorphic to $L_\calO^\lambda$. Hence, since the grading of $S^\calO$ is nonnegative, all negative graded components of $P_\calO^\lambda$ are zero. We deduce that $c_{\lambda\xi}^\calO(q)=[P_\calO^\lambda:L_\calO^\xi]_q\in\bbN[q]$.
\end{proof}

For each $\lambda,\xi\in\calP_{\bfd}^l$, we write
$c_{\lambda\xi}(q)=\sum_{g\in\bbZ}c_{\lambda\xi}^gq^g$ and $c_{\lambda\xi}^\calO(q)=\sum_{g\in\bbZ}c_{\lambda\xi}^{\calO g}q^g.$

\medskip

\subsection{Identification of grading of the KLR-algebra}
\label{subs_isom-gr-KLR}
We identify $\bfA\simeq \mod(S^\calO)$ as in Section \ref{subs_categ-equiv}. Under this identification, we can view the functor $E_\mathrm{S}$ from Theorem \ref{thm_equiv-Schur-O-ungr} as a functor $E_\mathrm{S}\colon \mod(S_{\bfd}^\bfs) \to \mod(S^\calO)$.
Recall that the algebra $S^\bfs_\bfd$ is equipped with an admissible grading, see Section \ref{subs_admis-grad}. We consider the graded basic algebra $^bS_{\bfd}^\bfs$ associated with $S_{\bfd}^\bfs$, i.e., we set
$$
^bS_{\bfd}^\bfs=\bigoplus_{\lambda,\xi\in \calP_{\bfd}^l}\Hom_{S_{\bfd}^\bfs}(P^\lambda,P^\xi)^{\rm op}=\bigoplus_{\lambda,\xi\in \calP_{\bfd}^l}\Hom_{R_{\bfd}^\bfs}(Y^\lambda,Y^\xi)^{\rm op},
$$
see Remark \ref{rk_full-faith-proj} (b).
The grading of $^bS^\bfs_\bfd$ is induced from the grading of the module $P^\lambda$, see Section \ref{subs_admis-grad}.
By definition, the categories $\mod(^bS_{\bfd}^\bfs)$ and $\mod(S_{\bfd}^\bfs)$ are equivalent.
Further, we have an ungraded algebra isomorphism $^bS_{\bfd}^\bfs\simeq S^\calO$, because the categories $\mod(^bS_{\bfd}^\bfs)$ and $\mod(S^\calO)$ are equivalent and both algebras $^bS_{\bfd}^\bfs$ and $S^\calO$ are basic.

Now, let $^bR_{\bfd}^\bfs$ be the graded basic algebra associated with $R_{\bfd}^\bfs$, i.e., we set
$$
^bR_{\bfd}^\bfs=\bigoplus_{\lambda,\xi\in \calK_{\bfd}^l}\Hom_{S_{\bfd}^\bfs}(P^\lambda,P^\xi)^{\mathrm{op}}=\bigoplus_{\lambda,\xi\in \calK_{\bfd}^l}\Hom_{R_{\bfd}^\bfs}(Y^\lambda,Y^\xi)^{\mathrm{op}}.
$$
The categories $\mod(^bR_{\bfd}^\bfs)$ and $\mod(R_{\bfd}^\bfs)$ are equivalent.
Further, we have an ungraded algebra isomorphism $^bR_{\bfd}^\bfs\simeq R^\calO$, because the categories $\mod(^bR_{\bfd}^\bfs)$ and $\mod(R^\calO)$ are equivalent via $E_H$ and both algebras $^bR_{\bfd}^\bfs$ and $R^\calO$ are basic. Let us prove that these algebras are also isomorphic as graded algebras.
To do so, we will construct a homogeneous basis of $R^\calO=\bigoplus_{\lambda,\xi\in\calK_{\bfd}^l} \End_{S^\calO}(P_\calO^\lambda,P_\calO^\xi)^{\mathrm{op}}$.

For each $\xi\in \calK_{\bfd}^l$ and each $g,s\in\bbN$, write 
\begin{align*}
\rad^g P_\calO^\xi/\rad^{g+1} P_\calO^\xi&=\bigoplus_{\lambda\in \calP_{\bfd}^l,s\in\bbN}L_\calO^\lambda\langle s\rangle^{\oplus c_{\xi\lambda}^{\calO{gs}}},\\
\rad^g P^\xi/\rad^{g+1} P^\xi&=\bigoplus_{\lambda\in \calP_{\bfd}^l,s\in\bbN}L^\lambda\langle s\rangle^{\oplus c_{\xi\lambda}^{{gs}}}.
\end{align*}
For each $\lambda\in \calK_{\bfd}^l$, $g,s\in\bbN$ and $t\in[1,c_{\xi\lambda}^{\calO gs}]$, we consider the chain of maps
$$
\rad^g P_\calO^\xi\to\rad^g P_\calO^\xi/\rad^{g+1} P_\calO^\xi\to L_\calO^\lambda\langle s\rangle,
$$
where the right hand arrow is the projection to the $t$th copy of $L_\calO^\lambda\langle s\rangle$. Since $P_\calO^\lambda$ is projective, there exists a homogeneous map $\theta_{\lambda\xi}^{gst}$ of degree $s$ such that the following diagram commutes
$$
\begin{CD}
P_\calO^\lambda @>\Id>>P_\calO^\lambda\langle s\rangle\\
@V\theta_{\lambda\xi}^{gst}VV  @VVV\\
\rad^g P_\calO^\xi @>>> L_\calO^\lambda\langle s\rangle.
\end{CD}
$$

\begin{lem}
\label{lem_basis-KLR-O}
The set $\Theta=\{\theta_{\lambda\xi}^{gst}; \lambda,\xi\in \calK_{\bfd}^l, g,s\in \bbN,t\in[1,c_{\xi\lambda}^{\calO gs}]\}$
is a basis of $R^\calO$.
\end{lem}
\begin{proof}
Fix $\lambda,\xi\in \calK^l_\bfd$. Set $\Theta_{\lambda\xi}=\{\theta_{\lambda\xi}^{gst}; g,s\in \bbN,t\in[1,c_{\xi\lambda}^{\calO gs}]\}$. Let $\theta=\sum_{g,s,t} a_{\lambda\xi}^{gst}\theta_{\lambda\xi}^{gst}$ with $a_{\lambda\xi}^{gst}\in\bfk$ be a nontrivial linear combination of elements of $\Theta_{\lambda\xi}$. Let $g_0$ be minimal such that $a_{\lambda\xi}^{g_0s_0t_0}\ne 0$ for some $s_0,t_0$. Then we have $\Im \theta\subset \rad^{g_0}P_\calO^\xi$ and the composition
$$
P_\calO^\lambda\stackrel{\theta}{\to}\rad^{g_0}P_\calO^\xi \to \rad^{g_0}P_\calO^\xi/\rad^{g_0+1}P_\calO^\xi\to L_\calO^\lambda\langle s_0\rangle
$$
is nonzero, where the right hand map is the projection on the $t_0$th copy of $L_\calO^\lambda\langle s_0\rangle$. Thus $\theta$ is not zero. Hence, the elements of the set $\Theta_{\lambda\xi}$ are linearly independent.
Now, to show that they form a basis of  $\Hom_{S^\calO}(P_\calO^\lambda,P_\calO^\xi)$, we count the dimension. We have
$$
|\Theta_{\lambda\xi}|=c^{\calO}_{\xi\lambda}(1)=[P_\calO^\xi:L_\calO^\lambda]_1=\dim\Hom_{S^\calO}(P_\calO^\lambda,P_\calO^\xi).
$$
\end{proof}
\begin{lem}
\label{lem_basis-rad-KLR}
Assume that $f\in \bbN$ and $\lambda,\xi\in\calK^l_\bfd$. The set $\Theta_{\lambda\xi}^{\geqslant f}=\{\theta_{\lambda\xi}^{gst}; g,s\in \bbN,g\geqslant f,t\in[1,c_{\xi\lambda}^{\calO gs}]\}$
is a basis of $\Hom_{S^\calO}(P_\calO^\lambda,\rad^f P_\calO^\xi)$.
\end{lem}
\begin{proof}
We must check that the image of a linear combination of elements of $\Theta_{\lambda\xi}$ that contains an element of $\Theta_{\lambda\xi}\backslash\Theta_{\lambda\xi}^{\geqslant f}$ with a nonzero coefficient is not in $\rad^{f} P^\xi$. Let
$$
\theta=\sum_{g,s,t} a_{\lambda\xi}^{gst}\theta_{\lambda\xi}^{gst},~\quad a_{\lambda\xi}^{gst}\in\bfk,
$$
be such a linear combination. Let $g_0$ be minimal such that $a_{\lambda\xi}^{g_0s_0t_0}\ne 0$ for some $s_0,t_0$. We have $g_0<f$. We have $\Im \theta\subset \rad^{g_0}P_\calO^\xi$. In the same way as in the proof of Lemma \ref{lem_basis-KLR-O} we can show that the composition
$$
P_\calO^\lambda\stackrel{\theta}{\to}\rad^{g_0}P_\calO^\xi \to \rad^{g_0}P_\calO^\xi/\rad^{g_0+1}P_\calO^\xi
$$
is nonzero. Thus, we have $\Im\theta\not\subset \rad^{g_0+1}P_\calO^\xi$. This implies $\Im\theta\not\subset \rad^f P_\calO^\xi$ because $f\geqslant g_0+1$.
\end{proof}
\begin{rk}
\label{rem_basis}
We construct a homogeneous basis for $^bR_{\bfd}^\bfs$ in a similar way. It satisfies an analogue of Lemma \ref{lem_basis-rad-KLR}.
\end{rk}

\begin{lem}
\label{lem_P-rigid}
For each $\lambda\in \calK_{\bfd}^l$, the graded $S^\calO$-module $P_\calO^\lambda$ is very rigid.
\end{lem}
\begin{proof}
Recall the projective module $T\in \bfA$ and the functor $F'\colon \bfA\to \mod(R_\bfd^ \bfs)$ defined in Section \ref{subs_KLR-O}. The algebra $R_\bfd^{\bfs}$ is Frobenius by the theorem in the introduction of \cite{MM} and the functor $F'$ is $0$-faithful by \cite[Thm.~6.6]{R}. So, by \cite[Lem.~2.14]{RSVV}, the module $T$ is tilting. Thus it is injective. This implies that, for each $\lambda\in \calK_\bfd^l$, the module $P^\lambda_\calO$ is injective because it is a direct factor of $T$, see Section \ref{subs_KLR-O}. In particular it is autodual with respect to the BGG duality.

Now, by construction, we have $P_\calO^\lambda/\rad P_\calO^\lambda=L_\calO^\lambda$. Since $P_\calO^\lambda$ is autodual, we deduce that $\soc P_\calO^\lambda$ is irreducible. Thus $P_\calO^\lambda$ is very rigid by Lemma \ref{lem_rigidity}.
\end{proof}
Recall that $c_{\xi\lambda}^{\calO s}$ is the multiplicity of $L_\calO^\lambda\langle s \rangle$ in $P_\calO^\xi$ and that $c_{\xi\lambda}^{\calO gs}$ is the multiplicity of $L_\calO^\lambda\langle s \rangle$ in $\rad^gP_\calO^\xi/\rad^{g+1}P_\calO^\xi$.
\begin{coro}
\label{coro_basis-d=s}
For each $\lambda\in\calP^l_\bfd,\xi\in\calK_{\bfd}^l$ and each $g,s\in\bbN$, we have $c_{\xi\lambda}^{\calO gs}=c_{\xi\lambda}^{\calO s}$ if  $g=s$ and $c_{\xi\lambda}^{\calO gs}=0$ else.
\qed
\end{coro}
Further, since $S^\calO$ and $S^\bfs_\bfd$ are Morita equivalent, we have also the following.
\begin{coro}
\label{coro_P-rigid}
For each $\lambda\in \calK_{\bfd}^l$, the $S_{\bfd}^\bfs$-module $P^\lambda$ is rigid. \qed
\end{coro}
Now, we have the following.
\begin{lem}
\label{lem_dec-num-pos}
For each $\lambda,\xi\in\calP^l_\bfd$ we have $c_{\lambda\xi}(q)\in \delta_{\lambda\xi}+q\bbN[q]$.
\end{lem}
\begin{proof}
The statement follows from the condition $\rm (a)$ in Definition \ref{def_adm-grad} and the graded BGG-resiprocity.
\end{proof}
\begin{coro}
\label{coro_bas-Sch-pos}
We have $(^bS^\bfs_\bfd)_n=0$ for each $n<0$ and the algebra $(^bS^\bfs_\bfd)_0$ is semisimple.
\end{coro}
\begin{proof}
The first claim follows from Lemma \ref{lem_dec-num-pos} because $\dim_q\Hom_{S^\bfs_\bfd}(P^\lambda,P^\xi)=c_{\xi\lambda}(q)$. The second claim is true because by Lemma \ref{lem_dec-num-pos} we have $(^bS^\bfs_\bfd)_0=\bigoplus_{\lambda\in\calP^l_\bfd}\bfk\Id_{P^\lambda}$.
\end{proof}
The grading of each simple module $L^\lambda$, viewed as an $^bS^\bfs_\bfd$-module, is concentrated in degree $0$. Hence, since $(^bS^\bfs_\bfd)_0$ is semisimple by Corollary \ref{coro_bas-Sch-pos}, we have the following.
\begin{coro}
\label{coro-Ext1gr=0}
For each $\lambda,\xi\in\calP^l_\bfd$ we have $\Ext^1_{\grmod(S^\bfs_\bfd)}(L^\lambda,L^\xi)=0$. \qed
\end{coro}
Let $\xi\in\calP_{\bfd}^l$. Consider the filtration of $P^\xi$ given by
$$
P^\xi=P^\xi(0)\supset P^\xi(1)\supset P^\xi(2)\supset\cdots,
$$
where $P^\xi(f)=\sum_\theta\Im \theta$, the map $\theta$ runs over the set of homogeneous maps in $\Hom_{S^\bfs_\bfd}(P^\lambda,P^\xi)$ of degree $\geqslant f$ (for the grading in Section \ref{subs_admis-grad}) and $\lambda$ varies in $\calP^l_\bfd$.
Note that $P^\xi(f)$ is zero for sufficiently large $f$.

\begin{lem}
\label{lem_new-filtration}
Let $\lambda,\xi\in\calP_{\bfd}^l$ and $g,s\in\bbN$. The following hold.

{\rm(a)} if $s<g$ then $L^\lambda\langle s\rangle$ is not a constituent of $P^\xi(g)$,

{\rm(b)} if $s\geqslant g$ then $L^\lambda\langle s\rangle$ is not a constituent of $P^\xi/P^\xi(g)$.
\end{lem}

\begin{proof}
Part (a) follows from the fact that $c_{\xi\lambda}(q)=[P^\xi:L^\lambda]_q$ is in $\bbN[q]$, see Lemma \ref{lem_dec-num-pos}. Now, we prove (b).
Suppose that there exists a submodule $P'^\xi\subset P^\xi$ containing $P^\xi(g)$ such that there exists a nonzero homogeneous morphism $P'^\xi/P^\xi(g)\to L^\lambda\langle s\rangle$. By the projectivity of $P^\lambda$, the canonical map $P^\lambda\langle s\rangle\to L^\lambda\langle s\rangle$ factors as follows
$$
P^\lambda\langle s\rangle\to P'^\xi\to P'^\xi/P^\xi(g)\to  L^\lambda\langle s\rangle.
$$
The left arrow yields a morphism $\theta\in\Hom_{S_{\bfd}^\bfs}(P^\lambda,P'^\xi)$ of degree $s$ such that $\Im\theta\not\subset P^\xi(g)$. This contradicts to the definition of $P^\xi(g)$.
\end{proof}

\begin{lem}
\label{coro_filtr-comp}
For each $f\in\bbN$ and each $\xi\in\calP^l_\bfd$, we have $\rad^f P^\xi\subset P^\xi(f)$.
\end{lem}
\begin{proof}
By Lemma \ref{lem_new-filtration}, for each $f\in\bbN$, any simple subquotient of $P^\xi(f)/P^\xi(f+1)$ is of the form $L^\lambda\langle f \rangle$ for some $\lambda\in\calP^l_\bfd$. Thus, by Corollary \ref{coro-Ext1gr=0}, the module $P^\xi(f)/P^\xi(f+1)$ is semisimple. This implies the statement.
\end{proof}

\begin{lem}
\label{lem_coeff-triang}
For each $\lambda,\xi\in\calK_{\bfd}^l$ and each $g, s\in\bbN$ such that $s<g$, we have $c^{gs}_{\xi\lambda}=0$.
\end{lem}
\begin{proof}
By Lemma \ref{lem_new-filtration} and Corollary \ref{coro_filtr-comp} each simple subquotient of $\rad^gP^\xi$ is of the form $L^\lambda\langle f \rangle$ for some $\lambda\in\calP^l_\bfd$ and $f\geqslant g$. This implies the statement.
\end{proof}

\begin{coro}
\label{lem_P-semirigid}
For each $\lambda,\xi\in\calK_{\bfd}^l$ we have $[\mathrm{rad}^gP^\xi/\mathrm{rad}^{g+1}P^\xi:L^\lambda]_q=c_{\xi\lambda}^gq^g$.
\end{coro}
\begin{proof}
Assume that the statement is false. Let $g$ be maximal such that the statement fails. Thus, we have
\begin{eqnarray*}
[P^\xi/\mathrm{rad}^{g+1}P^\xi:L^\lambda]_q&=&c_{\xi\lambda}(q)-\sum_{t\geqslant g+1}[\mathrm{rad}^{t}P^\xi/\mathrm{rad}^{t+1}P^\xi:L^\lambda]_q\\
&=&c_{\xi\lambda}(q)-\sum_{t\geqslant g+1}c_{\xi\lambda}^{t}q^t\\
&=&\sum_{t=0}^gc_{\xi\lambda}^tq^t.
\end{eqnarray*}
Here the second equality follows from the maximality of $g$.
By Lemma \ref{lem_coeff-triang}, each power of $q$ in $[\rad^g P^\xi/\rad^{g+1}P^\xi:L^\lambda]_q$ is $\geqslant g$. We deduce that $[\mathrm{rad}^gP^\xi/\mathrm{rad}^{g+1}P^\xi:L^\lambda]_q={c'}_{\xi\lambda}^gq^g$ for some ${c'}_{\xi\lambda}^g< {c}_{\xi\lambda}^g$. Now, we have
$$
{c'}_{\xi\lambda}^g=[\mathrm{rad}^gP^\xi/\mathrm{rad}^{g+1}P^\xi:L^\lambda]_1=c^{\calO g}_{\xi\lambda}={c}_{\xi\lambda}^g.
$$
Here the second equality is Lemma \ref{lem_P-rigid}.
We come to a contradiction.
\end{proof}
Since $[\mathrm{rad}^gP^\xi/\mathrm{rad}^{g+1}P^\xi:L^\lambda]_q=\sum_{s\in\bbN} c_{\xi\lambda}^{gs}q^s$, we deduce from Corollary \ref{lem_P-semirigid} the following.
\begin{coro}
\label{coro_g=s}
For each $\lambda,\xi\in\calK^l_\bfd$ we have
$c_{\xi\lambda}^{gs}=c_{\xi\lambda}^{s}$ if  $g=s$ and $c_{\xi\lambda}^{gs}=0$ else.\qed
\end{coro}
Recall the equivalence of categories $E_S$ defined in Theorem \ref{thm_equiv-Schur-O-ungr}. Recall also the definition of the elements $\theta_{\lambda\xi}^{gst}\in\Hom_{\bfA}(P^\lambda_\calO,P^\xi_\calO)$. Set
$$
v_{\lambda\xi}^{gst}=E_S^{-1}(\theta_{\lambda\xi}^{gst})\in \Hom_{S_{\bfd}^\bfs}(P^\lambda,P^\xi)\subset {^bR}_{\bfd}^\bfs.
$$
The set $\{v_{\lambda\xi}^{gst};\lambda,\xi\in\calK^l_\bfd,g,s\in\bbN,t\in [1,c_{\xi\lambda}^{gs}]\}$ is a basis of $^bR_{\bfd}^\bfs$ because $E_S$ is an equivalence of categories. We have $\Im v_{\lambda\xi}^{gst}\subset \rad^g P^\xi$ and $\Im v_{\lambda\xi}^{gst}\not\subset \rad^{g+1} P^\xi$. Thus the analogue of Lemma \ref{lem_basis-rad-KLR} mentioned in Remark \ref{rem_basis} and the Corollary \ref{coro_g=s} imply that the minimal nonzero homogeneous component in $v_{\lambda\xi}^{gst}$ has degree $g$. Denote this homogeneous component of degree $g$ by ${v'}_{\lambda\xi}^{gst}$.

Consider the $\bfk$-linear map $\Phi\colon R^\calO\to R_{\bfd}^\bfs$ such that $\theta_{\lambda\xi}^{gst}\mapsto {v'}_{\lambda\xi}^{gst}$.
\begin{thm}
\label{thm-isom-grad-KLR}
The map $\Phi$ is a graded $\bfk$-algebra isomorphism.
\end{thm}
\begin{proof}
For each $f\in\bbN$ set $J^{>f}=\bigoplus_{n>f}(^bR^\bfs_\bfd)_n$.
By construction, for each $\theta_{\lambda\xi}^{gst}\in\Theta$ we have $\Phi(\theta_{\lambda\xi}^{gst})\equiv v_{\lambda\xi}^{gst} ~\mod~J^{>g}$.
This implies that for each homogeneous element $\theta\in {^bR^l_\bfd}$
we have $\Phi(\theta)\equiv E_S^{-1}(\theta)~\mod ~J^{>\deg\theta}$.
Now, for each $\theta_{\lambda\xi}^{gst},\theta_{\lambda'\xi'}^{g's't'}\in\Theta$ we have
\begin{align*}
\Phi(\theta_{\lambda\xi}^{gst})\Phi(\theta_{\lambda'\xi'}^{g's't'})&\equiv v_{\lambda\xi}^{gst}v_{\lambda'\xi'}^{g's't'}\\
&\equiv E_S^{-1}(\theta_{\lambda\xi}^{gst}\theta_{\lambda'\xi'}^{g's't'})\\
&\equiv\Phi(\theta_{\lambda\xi}^{gst}\theta_{\lambda'\xi'}^{g's't'})~\mod~J^{>g+g'}.
\end{align*}
On the other hand the map $\Phi$ is homogeneous of degree zero. Thus $\Phi$ is a graded algebra isomorphism.
\end{proof}

\medskip

\subsection{Identification of grading of the Schur algebra}
The goal of this section is to show that the algebras $^bS_{\bfd}^\bfs$ and $S^\calO$ are isomorphic as graded algebras.

\begin{lem}
\label{lem_Young-indec}
For each $\lambda\in\calP^l_\bfd$, the $R^\bfs_\bfd$-module $Y^\lambda$ is indecomposable.
\end{lem}
\begin{proof}
The $S^\bfs_\bfd$-module $P^\lambda$ is indecomposable because it has simple top $L^\lambda$. Thus the ring $\End_{S^\bfs_\bfd}(P^\lambda)$ is local. By Remark \ref{rk_full-faith-proj} (b) we have $\End_{R^\bfs_\bfd}(Y^\lambda)=\End_{S^\bfs_\bfd}(P^\lambda)$. Thus the module $Y^\lambda$ is also indecomposable.
\end{proof}

By Theorem \ref{thm-isom-grad-KLR} we can identify the graded categories $\grmod(R_{\bfd}^\bfs)\simeq \grmod(R^\calO)$.
For each $\lambda\in\calP_{\bfd}^l$, there is an integer $a_\lambda$ such that $Y^\lambda=Y_\calO^\lambda\langle a_\lambda\rangle$, because the modules $Y^\lambda$ and $Y_\calO^\lambda$ are indecomposable and are isomorphic as ungraded modules, see the construction of the functor $E_H$ and Lemma \ref{lem_Young-indec}.

\begin{lem}
For each $\lambda,\xi\in \calP_{\bfd}^l$ we have $a_\lambda=a_\xi$.
\end{lem}
\begin{proof}
We have
\begin{eqnarray*}
c_{\lambda\xi}(q) & = & [P^\lambda:L^\xi]_q\\
& = & \dim_q\Hom_{S_{\bfd}^\bfs}(P^\xi,P^\lambda)\\
& = & \dim_q\Hom_{R_{\bfd}^\bfs}(Y^\xi,Y^\lambda)\\
& = & \dim_q\Hom_{R^\calO}(Y_\calO^\xi\langle a_\xi\rangle,Y_\calO^\lambda\langle a_\lambda\rangle)\\
& = & \dim_q\Hom_{S^\calO}(P_\calO^\xi\langle a_\xi\rangle,P_\calO^\lambda\langle a_\lambda\rangle)\\
& = & q^{a_\lambda-a_\xi}[P_\calO^\lambda:L_\calO^\xi]_q\\
& = & q^{a_\lambda-a_\xi}c_{\lambda\xi}^\calO(q)
\end{eqnarray*}
So we have also $c_{\xi\lambda}(q)=q^{a_\xi-a_\lambda}c_{\xi\lambda}^\calO(q)$.
But the matrix $c(q)$ is symmetric by Corollary \ref{coro_mat-sym}. Thus, for each $\lambda,\xi$ in $\calP_{\bfd}^l$ we have $c_{\lambda\xi}(q)=q^{2(a_\xi-a_\lambda)}c_{\lambda\xi}(q)$. Thus $a_\lambda=a_\xi$ whenever $c_{\lambda\xi}(q)\ne 0$. To conclude it is enough to note that for each $\lambda,\xi\in\calP_{\bfd}^l$ we can find a chain $\lambda, \lambda_1, \lambda_2,\cdots, \lambda_n, \xi\in\calP_{\bfd}^l$ such that
$$
c_{\lambda\lambda_1}(q)\ne 0, c_{\lambda_1\lambda_2}(q)\ne 0,\cdots, c_{\lambda_{n-1}\lambda_{n}}(q)\ne 0, c_{\lambda_n\xi}(q)\ne 0
$$
because the algebra
$
S_{\bfd}^\bfs
$
is indecomposable, see Section \ref{subs_def-Hecke-Schur}.
\end{proof}

\begin{thm}
\label{thm_main-theorem}
There exists a graded algebra isomorphism $^bS_{\bfd}^\bfs\simeq S^\calO$.
\end{thm}
\begin{proof}
We have a chain of graded algebra isomorphisms
\begin{eqnarray*}
(^bS_{\bfd}^\bfs)^\mathrm{op} & \simeq & \bigoplus_{\lambda,\xi\in \calP_{\bfd}^l}\Hom_{S_{\bfd}^\bfs}(P^\xi,P^\lambda)\\
& \simeq & \bigoplus_{\lambda,\xi\in \calP_{\bfd}^l}\Hom_{R_{\bfd}^\bfs}(Y^\xi,Y^\lambda)\\
& \simeq & \bigoplus_{\lambda,\xi\in \calP_{\bfd}^l}\Hom_{R^\calO}(Y_\calO^\xi\langle a_\xi\rangle,Y_\calO^\lambda\langle a_\lambda\rangle)\\
& \simeq & \bigoplus_{\lambda,\xi\in \calP_{\bfd}^l}\Hom_{S^\calO}(P_\calO^\xi,P_\calO^\lambda)\\
& \simeq & (S^\calO)^\mathrm{op}.
\end{eqnarray*}
\end{proof}

\appendix

\section{Graded decomposition numbers}
\label{app_mult}
The goal of this appendix is to calculate the graded decomposition numbers in the affine parabolic category $\calO$ needed to prove Lemma \ref{lem_mult-O-Schur-coin}.

We keep the notation from Sections \ref{subs_Lie-alg}, \ref{subs_catA}. Let $e$ be an integer $>0$. Let $\bfP$ be the set of proper subsets of $\widehat \Phi$. We will refer to elements of $\bfP$ as \emph{parabolic types}. Suppose $\nu\in\bfP$. Let $\widehat\frakp_\nu$ be the unique parabolic subalgebra containing $\widehat\frakb$ whose set of roots is generated by $\widehat\Phi\cup(-\nu)$.
We will say that a weight $\lambda\in\widehat P$ is $\nu$-dominant (resp. $\nu$-antidominant) if $\lambda(\alpha_k^\vee)\geqslant 0$ for each $\alpha_k\in\nu$ (resp.  if $\lambda(\alpha_k^\vee)\leqslant 0$ for each $\alpha_k\in\nu$).

Let $\calO^\nu$ be the category of $\widehat\frakg$-modules such that $M=\bigoplus_{\lambda\in\widehat \frakh^*}M_\lambda$ with  $M_\lambda=\{m\in M; xm=\lambda(x)m, \forall x\in\widehat\frakh\}$ and $U(\widehat\frakp_\nu)m$ is finite dimensional for each $m\in M$.

Let $W$ be the Weyl group of $\widehat \frakg$. Let $\leqslant$ be the Bruhat order on $W$. For each $\lambda\in\widehat P$, $w\in W$ we set $w\cdot \lambda=w(\lambda+\widehat\rho)-\widehat\rho$. Let $S$ be the set of simple reflection in $W$. Fix $\mu\in\bfP$. Let $W_\mu\subset W$ be the parabolic subgroup generated by the simple reflection corresponding to $\mu$. Let $^\mu W\subset W$ be the set of minimal length representatives of the cosets in $W_\mu\backslash W$. Let $o_{\mu,-}\in\widehat P$ (resp. $o_{\mu,+}\in\widehat P$) be a weight of the form $(z_\lambda,\lambda,-e-N)$ (resp. a weight of the form $(z_\lambda,\lambda,e-N)$) with $\lambda\in P$ such that $o_{\mu,-}+\widehat\rho$ is antidominant (resp. $o_{\mu,+}+\widehat\rho$ is dominant) and the stabiliser of $o_{\mu,-}$ (resp. $o_{\mu,+}$) in $W$ under the dot-action is equal to $W_\mu$. Here the notation is as in Section \ref{subs_catA}. Let $\leqslant$ be the partial order on $\widehat P$ given by $\lambda_1\leqslant\lambda_2$ if $\lambda_2-\lambda_1$ is a linear combination of simple roots with coefficients in $\bbN$.

Let $\calO_{\mu,\pm}^\nu\subset \calO^\nu$ be the full subcategory consisting of the modules such that the highest weight of their simple subquotients are conjugated to $o_{\mu,\pm}$ under the dot-action.
For a $\nu$-dominant weight $\lambda\in\widehat P$, let $V^\nu(\lambda)$ be the parabolic Verma module with highest weight $\lambda$. We write just $V(\lambda)$ if $\nu=\emptyset$. Let $L(\lambda)$ be denote the unique simple quotient of $V(\lambda)$.

Let $I_{\mu,-}$ (resp. $I_{\mu,+}$) be the set of shortest (resp. longest) representatives of the cosets $W/W_\mu$ in $W$ and $I^\nu_{\mu,-}\subset I_{\mu,-}$ (resp. $I^\nu_{\mu,+}\subset I_{\mu,+}$) be the subset of elements $w$ such that $w\cdot o_{\mu,-}$ (resp. $w\cdot o_{\mu,+}$) is $\nu$-dominant.

Fix $v\in I_{\mu,-}^\nu$ and $u\in I_{\mu,+}^\nu$.
Set
$$
^vI_{\mu,-}^\nu=\{w\in I_{\mu,-}^\nu; w\leqslant v\},\quad
^uI_{\mu,+}^\nu=\{w\in I_{\mu,+}^\nu; w\leqslant u\}.
$$
\begin{rk}
There is a bijection $I_{\mu,-}^\nu\to I_{\nu,+}^\mu,~x\mapsto x^{-1}$. If $v\in I_{\mu,-}^\nu$ and $u=v^{-1}\in I_{\nu,+}^\mu$, then there is a bijection $^vI_{\mu,-}^\nu\to {^uI_{\nu,+}^\mu},~x\mapsto x^{-1}$. See \cite[Cor.~3.3]{SVV2}.
\end{rk}

Let $^{v}\calO_{\mu,-}^\nu$ be the category consisting of finitely generated objects in the Serre subcategory of $\calO_{\mu,-}^\nu$ generated by the modules $L(w\cdot o_{\mu,-})$ with $w\in {^vI_{\lambda,-}^\nu}$.
Let $^{u}\calO_{\mu,+}^\nu$ be the subcategory of finitely generated objects in the quotient of $\calO_{\mu,+}^\nu$ by the Serre subcategory generated by the modules $L(w\cdot o_{\mu,+})$ with $w\in {I_{\mu,+}^\nu}\backslash{^uI_{\mu,+}^\nu}$.

For each $x\in I_{\mu,-}^\nu$ and $y\in I_{\mu,+}^\nu$, let $^vP(x\cdot o_{\mu,-})$ and $^uP(y\cdot o_{\mu,+})$ be the projective covers of $L(x\cdot o_{\mu,-})$ and $L(y\cdot o_{\mu,+})$ in $^v\calO_{\mu,-}^\nu$ and $^u\calO_{\mu,+}^\nu$ respectively. Set
$$
^vP_{\mu,-}^\nu=\bigoplus_{x\in {^vI_{\mu,-}^\nu}}(^vP(x\cdot o_{\mu,-})),\quad ^uP_{\mu,+}^\nu=\bigoplus_{x\in {^uI_{\mu,+}^\nu}}(^uP(x\cdot o_{\mu,+})).
$$
Set $^vA_{\mu,-}^\nu=\End(^vP_{\mu,-}^\nu)^\mathrm{op}$ and $^uA_{\mu,+}^\nu=\End(^uP_{\mu,+}^\nu)^\mathrm{op}$.
We have the equivalences of categories $^{v}\calO_{\mu,-}^\nu\simeq \mod(^vA_{\mu,-}^\nu)$ and $^{u}\calO_{\mu,+}^\nu\simeq \mod(^uA_{\mu,+}^\nu).
$
For each $w\in {^uI_{\mu,+}^\nu}$, let $^uV^\nu(w\cdot o_{\mu,+})$ be the image of $V^\nu(w\cdot o_{\mu,+})$ in $^{u}\calO_{\mu,+}^\nu$. The categories $^v\calO_{\mu,-}^\nu$, $^u\calO_{\mu,+}^\nu$ are highest weight categories with standard modules $\{V^\nu(x\cdot o_{\mu,-}); x\in {^vI_{\mu,-}^\nu}\}$ and $\{^uV^\nu(x\cdot o_{\mu,+}); x\in {^uI_{\mu,+}^\nu}\}$ respectively.
For each $v\in I_{\mu,\pm}^\nu$ and each $x\in {^vI_{\mu,\pm}^\nu}$, let $^vV^{\nu}(x\cdot o_{\mu,\pm})^\vee$ be the costandard object in $^v\calO_{\mu,\pm}^\nu$ with simple socle $L(x\cdot o_{\mu,\pm})$. We will refer to them as \emph{dual Verma modules}.

By \cite[Thm.~3.12]{SVV2} the rings $^vA_{\mu,-}^\nu$ and $^uA_{\mu,+}^\nu$ admit Koszul gradings. Thus the categories $^{v}\calO_{\mu,-}^\nu$, $^{u}\calO_{\mu,+}^\nu$ admit Koszul gradings. Denote by $^{v}\overline\calO_{\mu,-}^\nu$, $^{u}\overline\calO_{\mu,+}^\nu$ their graded versions, i.e., we have $^{v}\overline\calO_{\mu,-}^\nu=\grmod(^vA_{\mu,-}^\nu)$, $^{u}\overline\calO_{\mu,+}^\nu=\grmod(^uA_{\mu,+}^\nu)$. All simple modules, Verma modules, dual Verma modules and projective modules in $^{v}\calO_{\mu,-}^\nu$, $^{u}\calO_{\mu,+}^\nu$ have graded lifts, see \cite{MO}. This graded lifts are unique up to a shift of the grading because all mentioned modules are indecomposable. Fix the graded lifts of simple modules concentrated in degree zero and fix the graded lifts of projective modules, Verma modules and dual Verma modules that satisfy the same normalization conditions as in Section \ref{subs_graded-hw-cat}.
In the notation above we will skip the upper index $\nu$ if $\nu=\emptyset$.
\begin{rk}
\label{rem_par-grad-adapt}
The graded ring $^vA_{\mu,\pm}^\nu$ is a factor of $^vA_{\mu,\pm}$ by a homogeneous ideal. See \cite[Lem.~5.12]{SVV2} for a proof in the positive level case. The negative level case is similar. In particular, for $v\in I_{\mu,\pm}^\nu$ we get the inclusion of Grothendieck groups $[^v\overline\calO_{\mu,\pm}^\nu]\subset [^v\overline\calO_{\mu,\pm}]$.
\end{rk}

\medskip

\subsection{Kazhdan-Lusztig polynomials}
In this section we follow the notation of Soergel \cite{Soe}. First, we define the Hecke algebra associated with $(W,S)$. All claims here are true for an arbitrary Coxeter system $(W,S)$. Let $l\colon W\to \bbN$ be the length function.
\begin{df}
The \emph{Hecke algebra} $\calH$ is the $\bbZ[q,q^{-1}]$-algebra generated by the elements $H_s$, $s\in S$, modulo the following defining relations
\begin{itemize}
    \item $H_s^2=1+(q^{-1}-q)H_s$\quad $\forall s\in S$,
    \item $H_sH_t\cdots H_s=H_tH_s\cdots H_t$ \quad  $\forall s,t\in S,~st\cdots s=ts\cdots t$,
    \item $H_sH_t\cdots H_t=H_tH_s\cdots H_s$ \quad  $\forall s,t\in S,~st\cdots t=ts\cdots s$.
\end{itemize}
\end{df}
For $w\in W$ with reduced expression $w=st\cdots r$ set $H_w=H_sH_t\cdots H_r$.
The algebra $\calH$ has a unique ring homomorphism $\overline\bullet\colon\calH\to\calH$, $H\mapsto \overline{H}$
such that $\overline{q}=q^{-1}$ and $\overline{H_w}=(H_{w^{-1}})^{-1}$.
Let $(\overline H_x)_{x\in W}$ be the Kazhdan-Lusztig basis of $\calH$, see \cite[Thm.~2.1]{Soe}.
We write $\underline{H}_x=\sum_{y\in W}h_{y,x}(q)H_y$, where $h_{y,x}(q)$ is a polynomial in $\bbZ[q]$ for each $x,y\in W$.
The polynomial $h_{x,y}(q)$ is nonzero only if $x\leqslant y$.

For each $x,y\in W$, let $h^{x,y}$ be the inverse Kazhdan-Lusztig polynomial, i.e., such that 
$$
\sum_{z\in W}(-1)^{l(x)+l(z)}h_{z,x}(q)h^{z,y}(q)=\delta_{x,y}.
$$

Fix a subset $f\subset S$ such that the subgroup $W_f\subset W$ generated by $f$ is finite. Let $^fW$ be the set of minimal length representatives for the cosets $W_f\backslash W$.

For each $x,y\in {^f W}$, let $n_{x,y}(q)$ be the parabolic Kazhdan-Lusztig polynomial associated with the parabolic subgroup $W_f\subset W$, i.e., we have
$$
n_{x,y}(q)=\sum_{z\in W_f}(-q)^{l(z)}h_{zx,y}(q).
$$
For each $x,y\in{^fW}$, let $n^{x,y}(q)$ be the inverse parabolic Kazhdan-Lusztig polynomial, i.e., we have 
$$
\sum_{z\in {^f W}}(-1)^{l(x)+l(z)}n_{z,x}(q)n^{z,y}(q)=\delta_{x,y}.
$$
We may write $n_{x,y}^f(q)$, $n_f^{x,y}(q)$ to insist on the parabolic type $f$.

Now, we recall some basic properties of Kazhdan-Lusztig polynomials.
\begin{lem}
\label{lem_polKL}
Let $x,y,z\in W$. We have

    {\rm(a)} $h_{x,y}(q)\in q^{l(y)-l(x)}(1+q^{-2}\bbZ[q^{-2}])$ if $x\leqslant y$,

    {\rm(b)} $h_{x,y}(q)=h_{x^{-1},y^{-1}}(q),~h^{x,y}(q)=h^{x^{-1},y^{-1}}(q)$,

    {\rm(c)} $n^{x,y}(q)=h^{x,y}(q)$ if $x,y\in {^f W}$,

    {\rm(d)} $h^{x,1}(q)=q^{l(x)}$,

    {\rm(e)} $n_{xz,y}(q)=q^{l(z)}n_{x,y}(q)$ if $x,y\in {^f W}$ and $l(xz)=l(x)-l(z), l(yz)=l(y)-l(z)$,

    {\rm(f)} $h_{x,y}(-q)=(-1)^{l(x)+l(y)}h_{x,y}(q)$.
\end{lem}
\begin{proof}
Part (a) follows from \cite[Rem.~2.6]{Soe} and \cite[Lem.~2.6]{KaLu}. Parts (c), (e) are respectively \cite[Prop.~3.7~(2)]{Soe} and \cite[Rem.~3.2~(4)]{Soe}. Part (b) is true because the $\bbZ[q,q^{-1}]$-algebra $\calH$ contains an antiautomorphism sending $H_w$ to $H_{w^{-1}}$ for each $w\in W$. The statement (f) follows from (a).

Let us prove (d). The algebra $\calH$ has a representation $\mathrm{sgn}$ on $\bbZ[q,q^{-1}]$ such that each $H_s$ with $s\in S$ acts by $-q$. We equip $\bbZ[q,q^{-1}]$ with the involution such that $\overline{p(q)}={p(q^{-1})}$ for each $p(q)\in \bbZ[q,q^{-1}]$. The $\calH$-action on $\bbZ[q,q^{-1}]$ is compatible with the involution. The element $\underline{H}_w$ for $w\in W,w\ne 1$, acts on $\bbZ[q,q^{-1}]$ by zero, because it acts by some polynomial $p(q)\in q\bbZ[q]$ such that $p(q)=p(q^{-1})$ because $\underline{H}_w\in H_w+q(\sum_{w'<w}\bbZ[q]H_{w'})$. Now, consider the equality
$$
H_x=\sum_{y\in W,y\leqslant x}(-1)^{l(x)+l(y)}h^{x,y}(q)\underline{H}_y.
$$
Applying $\mathrm{sgn}$ to both sides yields $h^{x,1}(q)=q^{l(x)}$.
\end{proof}

\medskip

\subsection{Plan of the proof}
Here we indicate the scheme that we follow to get the graded decomposition numbers in $^v\overline\calO_{\mu,-}^\nu$ and $^u\overline\calO_{\mu,+}^\nu$ (i.e., the graded multiplicities of simple module in parabolic Verma modules in sense of Section \ref{subs_graded-hw-cat}).
\begin{itemize}
\item[Step 1.] We compute the decomposition numbers of $^v\overline\calO_{\emptyset,-}$ and $v\in W$ using the geometric approach of \cite{Shan}.
\item[Step 2.] We compute the decomposition numbers of $^v\overline\calO_{\emptyset,-}^\nu$, for $\nu\in\bfP$ and $v\in I_{\emptyset,-}^\nu$ using the previous step and a graded version of the BGG-resolution.
\item[Step 3.]  We compute the decomposition numbers of $^u\overline\calO_{\nu,+}$, for $\nu\in\bfP$ and $u\in I_{\nu,+}$ using the previous step and the Koszul duality.
\item[Step 4.]  We compute the decomposition numbers of $^u\overline\calO_{\nu,+}^\mu$, for $\nu\in\bfP$ and $u\in I_{\nu,+}^\mu$ using the previous step and a graded version of the BGG-resolution.
\item[Step 5.] We compute the decomposition numbers of $^v\overline\calO_{\mu,-}^\nu$, for $\nu,\mu\in\bfP$ and $v\in I_{\mu,-}^\nu$ using the previous step and the Koszul duality.
\end{itemize}

\medskip

\subsection{Step 1}
\label{subs_reg-npar-neg}
Let $v\in W$.
We must count the graded multiplicities
$$
[V(x\cdot o_{\emptyset,-}):L(y\cdot o_{\emptyset,-})]_q=\sum_{g\in \bbZ}[V(x\cdot o_{\emptyset,-}):L(y\cdot o_{\emptyset,-})\langle g\rangle]q^g
$$
for each $x,y\in {^vI_{\emptyset,-}}$.
\begin{lem}
\label{lem_reg-npar-neg}
For each $x,y\in {^vI_{\emptyset,-}}$ we have $[V(x\cdot o_{\emptyset,-}):L(y\cdot o_{\emptyset,-})]_q=h^{x,y}(q)$.
\end{lem}
\begin{proof}

Let $s$ be the length of the socle filtration of $V(x\cdot o_{\emptyset,-})$, i.e., the integer $s$ such that $\soc^s V(x\cdot o_{\emptyset,-})=V(x\cdot o_{\emptyset,-})$ and $\soc^{s-1}V(x\cdot o_{\emptyset,-})\subsetneq V(x\cdot o_{\emptyset,-})$, see Section \ref{subs_ridid-modules}. By Lemma \ref{lem_rigidity} and \cite[Lem.~7.3]{Fie}, the grading filtration of $V(x\cdot o_{\emptyset,-})$ (viewed as an $^vA_{\emptyset,-}$-module) coincides with its socle filtration. Thus the graded multiplicity $[V(x\cdot o_{\emptyset,-}):L(y\cdot o_{\emptyset,-})]_q$ coincides with
$$
[V(x\cdot o_{\emptyset,-}):L(y\cdot o_{\emptyset,-})]^\soc_q=\sum_{g=1}^sq^{s-g}[\soc^g(V(x\cdot o_{\emptyset,-}))/\soc^{g-1}(V(x\cdot o_{\emptyset,-}):L(y\cdot o_{\emptyset,-})].
$$
Now, it suffices to show that $[V(x\cdot o_{\emptyset,-}):L(y\cdot o_{\emptyset,-})]^\soc_q=h^{x,y}(q)$.
This follows from \cite{Shan}. More precisely, by \cite[Prop.~4.2]{Shan} the Verma modules and the simple ones are in the essential image of the global section functor defined in \cite[Sec.~4.3]{Shan}. Moreover, by \cite[Thm.~7.15.6]{BD} and \cite[Thm.~5.5]{FrGa} the global section functor is fully faithful. Thus $[V(x\cdot o_{\emptyset,-}):L(y\cdot o_{\emptyset,-})]^\soc_q$ can be interpreted in terms of graded multiplicities for $D$-modules. This multiplicity can be computed explicitly via Kazhdan-Lusztig polynomials using the formulas \cite[(6.1),~(6.2)]{Shan} and the fact that the weight filtration in \cite[Sec.~6.1]{Shan} coincides with the socle filtration, see \cite[Lem.~5.2.2]{BB} for details.
\end{proof}

\medskip

\subsection{Step 2}
\label{subs_BGG-neg}
Let $\mu,\nu\in \bfP$ and $v\in I_{\mu,-}^\nu$.
\begin{lem}
Let $\lambda=w\cdot o_{\mu,-}$ with $w\in W$ and $s\in S$ such that $s\cdot\lambda<\lambda$. Assume that $v$ is large enough such that $V(\lambda)$, $V(s\cdot \lambda)\in{^v\calO_{\mu,-}}$. Then, each nonzero morphism $\theta\colon V(s\cdot\lambda)\to V(\lambda)$ in $^v\calO_{\mu,-}$ is homogeneous of degree $1$.
\end{lem}
\begin{proof}
We have $V(o_{\mu,-})=L(o_{\mu,-})$. It is a submodule of each Verma module $V(x\cdot o_{\mu,-})$, $x\in W$ because the weight $o_{\mu,-}+\widehat\rho$ is antidominant. By Lemma \ref{lem_reg-npar-neg} and Lemma \ref{lem_polKL} (d) we have $[V(\lambda):L(o_{\mu,-})]_q=q^{l(w)}$ and $[V(s\cdot\lambda):L(o_{\mu,-})]_q=q^{l(w)-1}$.
Thus $V(\lambda)$ and $V(s\cdot\lambda)$ contain a unique submodule isomorphic to $L(o_{\mu,-})$, which is concentrated in degree $l(w)$ and $l(w)-1$ respectively. Each nonzero morphism from $V(s\cdot\lambda)$ to $V(\lambda)$ is injective by \cite[Sec.~2]{KK}. Thus each morphism from $V(s\cdot\lambda)$ to $V(\lambda)$ is homogeneous of degree $1$.
\end{proof}

Let $v\in I_{\mu,-}^\nu$ and $\lambda=w\cdot o_{\mu,-}$, with $w\in {^vI_{\mu,-}^\nu}$.
See \cite[Chap.~6]{BGG} for the definition of the BGG resolution. Applying the exact functor $U(\frakg)\otimes_{U(\frakp)}\bullet$ to the BGG resolution of the simple module with highest weight $\lambda$ of the Levi subalgebra $\widehat\frakl_\nu\subset\widehat\frakp_\nu$, we get the parabolic BGG resolution of the parabolic Verma module $V^\nu(\lambda)$
$$
\cdots\to C_2\to C_1\to V(\lambda)\stackrel{\epsilon}{\to} V^\nu(\lambda)\to 0, \eqno (\mathrm{A}1)
$$
where
$$
C_t=\bigoplus_{x\in W_\nu, l(x)=t}V(x\cdot\lambda).
$$
\begin{coro}
\label{coro_graded-BGG-parab}
Each map except $\epsilon$ in $(\mathrm{A}1)$ is of degree $1$. The map $\epsilon$ is of degree $0$.
\end{coro}
\begin{proof}
We must explain why $\epsilon$ is homogeneous of degree $0$. By definition of the graded lifts of (parabolic) Verma modules, the morphisms $V(\lambda)\to L(\lambda)$ and  $V^\nu(\lambda)\to L(\lambda)$ are homogeneous of degree zero. The morphism $V(\lambda)\to V^\nu(\lambda)$ is homogeneous because $\dim\Hom_{\calO_{\mu,-}}(V(\lambda),V^\nu(\lambda))=1$. Its degree is automatically zero.
\end{proof}

Assume that $\nu\in\bfP$, $v\in I_{\emptyset,-}^\nu$.
\begin{coro}
\label{coro_reg-par-neg}
Assume that $v\in I_{\emptyset,-}^\nu$. For each $x,y\in {^vI_{\emptyset,-}^\nu}$ we have
$$
[V^\nu(x\cdot o_{\emptyset,-}):L(y\cdot o_{\emptyset,-})]_q=\sum_{z\in W_\nu}(-q)^{l(z)}h^{zx,y}(q).
$$
\end{coro}
\begin{proof}
By Corollary \ref{coro_graded-BGG-parab} we have
$$
[V^\nu(x\cdot o_{\emptyset,-})]=\sum_{z\in W_\nu}(-q)^{l(z)}[V(zx\cdot o_{\emptyset,-})]
$$
in $[^v\overline\calO_{\emptyset,-}]$.
Now, apply Lemma \ref{lem_reg-npar-neg}.
\end{proof}

\begin{coro}
\label{coro_inv-reg-par-neg}
For each $x\in {^vI_{\emptyset,-}^\nu}$ we have the following equality in $[^v\overline\calO_{\emptyset,-}^\nu]$:
$$
[L(x\cdot o_{\emptyset,-})]=\sum_{y\in {^v{I_{\emptyset,-}^\nu}}}(-1)^{l(x)+l(y)}h_{y^{-1},x^{-1}}(q)[V^\nu(y\cdot o_{\emptyset,-})].
$$
\qed
\end{coro}

\begin{proof}
We must prove that
$$
\sum_{t\in {^vI_{\emptyset,-}^\nu}}(-1)^{l(x)+l(t)}h_{t^{-1},x^{-1}}(q)[V^\nu(t\cdot o_{\emptyset,-}):L(y\cdot o_{\emptyset,-})]_q=\delta_{x,y}
$$
for each $x,y\in {^vI_{\emptyset,-}^\nu}$.
We have
\begin{eqnarray*}
&&\sum_{t\in {^vI_{\emptyset,-}^\nu}}(-1)^{l(x)+l(t)}h_{t^{-1},x^{-1}}(q)[V^\nu(t\cdot o_{\emptyset,-}):L(y\cdot o_{\emptyset,-})]_q\\
&=&\sum_{t\in {^vI_{\emptyset,-}^\nu},z\in W_\nu}(-q)^{l(z)}(-1)^{l(x)+l(t)}h_{t^{-1},x^{-1}}(q)h ^{zt,y}(q)\\
&=&\sum_{t\in {^vI_{\emptyset,-}^\nu},z\in W_\nu}(-1)^{l(x)+l(t)+l(z)}h_{t^{-1}z^{-1},x^{-1}}(q)h ^{zt,y}(q)\\
&=&\sum_{s\in {^vI_{\emptyset,-}}}(-1)^{l(x)+l(s)}h_{s^{-1},x^{-1}}(q)h^{s^{-1},y^{-1}}(q)\\
&=&\delta_{x,y}.
\end{eqnarray*}
Here the second equality is Lemma \ref{lem_polKL} (e), the third is Lemma \ref{lem_polKL} (b).
\end{proof}

\medskip

\subsection{Step 3}

Let $\nu\in\bfP$ and $u\in I_{\nu,+}$.
\begin{lem}
\label{lem_sing-npar-pos}
For each $x,y\in {^uI_{\nu,+}}$ we have
$$
[^uV(x\cdot o_{\mu,+}):L(y\cdot o_{\mu,+})]_q=h_{x,y}(q).
$$
\end{lem}
\begin{proof}
Set $v=u^{-1}$.
As in the proof of Corollary \ref{coro_grBGG2}, we deduce from Corollary \ref{coro_inv-reg-par-neg} that we have
$$
[L(x\cdot o_{\emptyset,-})]=\sum_{y\in {^v{I_{\emptyset,-}^\nu}}}(-1)^{l(x)+l(y)}h_{y^{-1},x^{-1}}(q^{-1})[V^{\nu}(y\cdot o_{\emptyset,-})^\vee]
$$
in $[^v\overline\calO_{\emptyset,-}^\nu]$.
By \cite[Thm.~3.12]{SVV2} and Corollary \ref{coro_Groth-isom-Koszul} we have a $\bbZ[q,q^{-1}]$-module isomorphism
$$
[^v\overline\calO_{\emptyset,-}^\nu]\to\widetilde{[^u\overline\calO_{\nu,+}]},\quad [L(x\cdot o_{\emptyset,-})]\mapsto[^uP(x^{-1}\cdot o_{\nu,+})],~[V^\nu(x\cdot o_{\emptyset,-})^\vee]\mapsto [^uV(x^{-1}\cdot o_{\nu,+})]~\forall x\in {^vI_{\emptyset,-}^\nu}.
$$
Thus we get
$$
[^uP(x^{-1}\cdot o_{\nu,+})]=\sum_{y\in {^vI_{\emptyset,-}^\nu}}(-1)^{l(x)+l(y)}h_{y^{-1},x^{-1}}(-q)[^uV({y^{-1}}\cdot o_{\nu,+})]
$$
in $[^u\overline\calO_{\nu,+}]$.
Now, the statement follows from Corollary \ref{coro_grBGG2} and Lemma \ref{lem_polKL} (f).
\end{proof}

\medskip

\subsection{Step 4}
\label{subs_BGG-pos}
Let $\mu,\nu\in\bfP$ and $u,w\in I_{\nu,+}^\mu$ with $w\leqslant u$, and set $\lambda=w\cdot o_{\nu,+}$. There is an exact sequence
$$
\cdots\to C_2\to C_1\to V(\lambda)\stackrel{\epsilon}{\to} V^\mu(\lambda)\to 0
$$
in $\calO_{\nu,+}$, where
$$
C_t=\bigoplus_{x\in W_\mu, l(x)=t}V(x\cdot\lambda),
$$
compare Section \ref{subs_BGG-neg}.
The image of this exact sequence in $^u\calO_{\nu,+}$ yields an exact sequence
$$
\cdots\to {^uC_2}\to {^uC_1}\to {^uV(\lambda)}\stackrel{\epsilon}{\to}{^uV^\mu(\lambda)}\to 0,  \eqno (\mathrm{A}2)
$$
where
$$
^uC_t=\bigoplus_{x\in W_\mu, l(x)=t}{^uV(x\cdot\lambda)}.
$$
\begin{lem}
\label{lem_morph-Verma-tr-inj}
Let $u\in I_{\nu,+}$ with $x,y\in {^uI_{\nu,+}}$ and $x\geqslant y$. Then

    {\rm(a)} $\dim\Hom_{^u\calO_{\nu,+}}(^uV(x\cdot o_{\nu,+}),{^uV(y\cdot o_{\nu,+})})=1$,

    {\rm(b)} each nonzero morphism $^uV(x\cdot o_{\nu,+})\to {^uV(y\cdot o_{\nu,+})}$ is injective,

    {\rm(c)} $\soc({^uV(y\cdot o_{\nu,+})})={^uV(u\cdot o_{\nu,+})}$,

    {\rm(d)} the inclusion ${^uV(u\cdot o_{\nu,+})}\to{^uV(y\cdot o_{\nu,+})}$ is homogeneous of degree $l(u)-l(y)$,

    {\rm(e)} each nonzero morphism $^uV(x\cdot o_{\nu,+})\to {^uV(y\cdot o_{\nu,+})}$ is homogeneous of degree $l(x)-l(y)$.
\end{lem}
\begin{proof}
By the Ringel equivalence we have
$$
\Hom_{^u\calO_{\nu,+}}(^uV(x\cdot o_{\nu,+}),{^uV(y\cdot o_{\nu,+})})=\Hom_{\calO_{\nu,-}}(V(y\cdot o_{\nu,-}),{V(x\cdot o_{\nu,-})}),
$$
see \cite[Prop.~3.9~(b)]{SVV2}.
So, to prove (a) it is enough to prove that
$$
\dim\Hom_{\calO_{\nu,-}}(V(y\cdot o_{\nu,-}),{V(x\cdot o_{\nu,-})})=1.
$$
This follows from \cite[Lem.~7.3]{Fie} and \cite[Sec.~2]{KK}, see the proof of \cite[Thm.~4.2~(b)]{BGG}.

Now, to prove (b) it is enough to show that there exists one injective morphism $^uV(x\cdot o_{\nu,+})\to{^uV(y\cdot o_{\nu,+})}$. By \cite[Prop.~3.1]{KaTan} there exist an injective morphism $V(x\cdot o_{\nu,+})\to{V(y\cdot o_{\nu,+})}$. Its image in $^u\calO_{\nu,+}$ by the quotient functor (which is exact) must be also injective.

It is clear from (b) that the only simple module that can be contained in $\soc({^uV(y\cdot o_{\nu,+})})$ is ${^uV(u\cdot o_{\nu,+})}=L(u\cdot o_{\nu,+})$. Hence, by (a), this socle must be equal to ${^uV(u\cdot o_{\nu,+})}$. Part (c) is proved.

By Lemma \ref{lem_rigidity}, the radical filtration of $^uV(y\cdot o_{\nu,+})$ coincides with its grading filtration. Here we view $^uV(y\cdot o_{\nu,+})$ as an $^uA_{\nu,+}$-module. By Lemma \ref{lem_polKL} (a) and Lemma \ref{lem_sing-npar-pos}, for all $z\in {^uI_{\nu,+}}$ such that $z\ne u$ and $z\leqslant y$ we have $[^uV(y\cdot o_{\nu,+}):{L(z\cdot o_{\nu,+})}]_q\in q^{l(u)-l(y)-1}\bbZ[q^{-1}]$ and $[^uV(y\cdot o_{\nu,+}):{L(u\cdot o_{\nu,+})}]_q\in q^{l(u)-l(y)}(1+q^{-1}\bbZ[q^{-1}])$. Thus $\rad^{l(u)-l(y)}(^uV(y\cdot o_{\nu,+}))$ coincides with $\soc(^uV(y\cdot o_{\nu,+}))$. Part (d) follows.

The statement (e) follows from (a), (b), (c), (d).

\end{proof}
\begin{coro}
\label{coro_graded-BGG-parab-pos}
Each map in {\rm (A2)} except $\epsilon$
is homogeneous of degree $1$. The map $\epsilon$ is homogeneous of degree $0$. \qed
\end{coro}

Assume that $\nu,\mu\in\bfP$, $u\in I_{\nu,+}^\mu$.
\begin{coro}
\label{coro_sing-par-pos}
For each $x,y\in {^uI_{\nu,-}^\mu}$ we have
$$
[^uV^\mu(x\cdot o_{\nu,+}):L(y\cdot o_{\nu,+})]_q=n^\mu_{x,y}(q).
$$
\end{coro}
\begin{proof}
By Corollary \ref{coro_graded-BGG-parab-pos} we have
$$
[^uV^\mu(x\cdot o_{\nu,+})]=\sum_{s\in W_\mu}(-q)^{l(s)}[^uV(sx\cdot o_{\nu,+})]
$$
in $[{^u\overline\calO_{\nu,+}}]$.
Now, from Lemma \ref{lem_sing-npar-pos} we get
$$
[^uV^\mu(x\cdot o_{\nu,+}):L(y\cdot o_{\nu,+})]_q=\sum_{s\in W_\mu}(-q)^{l(s)}h_{sx,y}(q)=n^\mu_{x,y}(q).
$$
\end{proof}

\medskip

\subsection{Step 5}
Assume that $\nu,\mu\in\bfP$ and $v\in I_{\nu,-}^\mu$.

\begin{lem}
\label{lem_inv-sing-par-neg}
For each $x\in {^vI_{\mu,-}^\nu}$ we have the following equality in $[^v\overline\calO_{\mu,-}^\nu]$
$$
[L(x\cdot o_{\mu,-})]=\sum_{y\in {^vI_{\mu,-}^\nu}}(-1)^{l(x)+l(y)}n^\mu_{y^{-1},x^{-1}}(q)[V^\nu(y\cdot o_{\mu,-})].
$$
\end{lem}
\begin{proof}
The statement follows from Corollary \ref{coro_sing-par-pos}, applying the Koszul duality in the same way as in the proof of Lemma \ref{lem_sing-npar-pos}.
\end{proof}


\begin{coro}
\label{coro_sing-par-neg}
For each $x,y\in {^vI_{\mu,-}^\nu}$ we have
$$
[V^\nu(x\cdot o_{\mu,-}):L(y\cdot o_{\mu,-})]_q=\sum_{z\in W_\nu}(-q)^{l(z)}h^{zx,y}(q).
$$
\end{coro}
\begin{proof}
We need to check that
$$
\sum_{t\in {^vI_{\mu,-}^\nu},z\in W_\nu}(-q)^{l(z)}(-1)^{l(x)+l(t)}n^\mu_{t^{-1},x^{-1}}(q)h ^{zt,y}(q)=\delta_{x,y}.
$$
This can be done in the same way as in the proof of Corollary \ref{coro_inv-reg-par-neg}.
\end{proof}

\section*{Acknowledgements}
I would like  to thank Eric Vasserot for his guidance and helpful discussions during my work on this paper.


\end{document}